\documentclass[12pt]{amsart}

\usepackage{amsmath, amssymb, amsthm, mathrsfs, graphicx, hyperref, float, mathabx}
\usepackage[shortlabels]{enumitem}
\usepackage[all,cmtip]{xy}
\usepackage[numbers]{natbib}
\usepackage[margin=1in]{geometry}
\usepackage{tikz-cd}
\usepackage{todonotes}
\usepackage{stmaryrd}
\usepackage{yfonts}

\graphicspath{{graphics/}}

\newtheorem{theorem}{Theorem}
\newtheorem{proposition}[theorem]{Proposition}
\newtheorem{lemma}[theorem]{Lemma}

\theoremstyle{definition}
\newtheorem{definition}[theorem]{Definition}
\newtheorem{example}[theorem]{Example}
\newtheorem{remark}[theorem]{Remark}

\numberwithin{theorem}{section}

\makeatletter
\providecommand{\leftsquigarrow}{%
  \mathrel{\mathpalette\reflect@squig\relax}%
}
\newcommand{\reflect@squig}[2]{%
  \reflectbox{$\m@th#1\rightsquigarrow$}%
}
\makeatother

\renewcommand{\phi}{\varphi} 

\definecolor{sebgreen1}{rgb}{0.019,0.317,0.149}
\definecolor{sebgreen2}{rgb}{0.784,0.952,0.780}
\definecolor{refred1}{rgb}{1,0,0}
\definecolor{refred2}{rgb}{1,.8,.8}

\newcommand{\mm}{\mathfrak{m}}

\newcommand{\A}{\mathbb{A}}

\newcommand{\N}{\mathbb{N}}
\newcommand{\PP}{\mathbb{P}}

\newcommand{\Z}{\mathbb{Z}}

\newcommand{\Spec}{\mathrm{Spec}}
\newcommand{\Tor}{\mathrm{Tor}}

\renewcommand{\O}{\mathcal{O}}
\newcommand{\res}{\mathrm{Res}}

\newcommand{\ord}{\mathrm{ord}}

   \usepackage{etoolbox}

   \makeatletter

\appto\maketitle{%
\let\@makefnmark\relax  \let\@thefnmark\relax
\ifx\@empty\addresses\else\@footnotetext{%
  \vskip-\bigskipamount\@setaddresses}
  }
\def\enddoc@text{}
\makeatother

   \makeatletter
   \patchcmd\maketitle
    {\uppercasenonmath\shorttitle}
     {}
    {}{}
    \patchcmd\maketitle
    {\@nx\MakeUppercase{\the\toks@}}
   {\the\toks@}
       {}
      {}{}
      \patchcmd\@setauthors
       {\MakeUppercase{\authors}}
     {\authors}
       {}{}
     \makeatother

\title{Residues and Gorenstein Contractions of Genus One Curves}
\author{Adrian Neff \\
Appendix by Adrian Neff and Jonathan Wise}
\address{Department of Mathematics, Ohio State University, 434 Math Tower, 231 West 18th Avenue, Columbus, OH 43210-1174}
\email{neff.314@osu.edu}

\begin{document}

\begin{abstract}
    Let $C$ be a genus one nodal curve over a local artinian base and let $E$ be a proper subcurve of genus one. We define residues for curves over local artinian rings, then define generalized residues with values in line bundles over the local artinian ring that arise from tropical data on the curve. We then use these residues to construct a contraction of $C$ that collapses $E$ to a Gorenstein genus one singularity. 
\end{abstract}

\maketitle

\tableofcontents

\section{Introduction}

Given a genus one nodal curve $C$ and a proper subcurve of genus one, $E$, we wish to construct a contraction of $C$ collapsing $E$ to a Gorenstein genus one singularity. Our construction begins with a family of centrally aligned genus one log curves $\pi : C \to S = \Spec(A)$, where $A$ is a local artinian ring. One may think of the central alignment as a semistable modification $\tilde{C} \to C$ that ``orders the components" inside the subcurve $E$. Using this data, we produce a residue for singular curves valued in $A$, then use this to construct the contraction $C \to \widebar{C}$. 

Specifically, let $C \to \Spec(A)$ be a family of curves over a local artinian ring $A$ and let $p \in C_0$ be a closed point in the fiber. We first construct the residue of a meromorphic differential $\phi$ along a branch $X$ at $p$, denoted by $\res_{p,X}(\phi)$, that satisfies the following two properties:
\begin{enumerate}
    \item if $p$ is a node with branches $X$ and $Y$, then 
    \[
    \res_{p,X}(\phi) = - \res_{p,Y}(\phi);
    \]
    \item if $C$ is smooth, then 
    \[
    \sum_{p \in C} \res_{p,C}(\phi) = 0.
    \]
\end{enumerate}
One may think of this as an algebraic analogue of integration over a vanishing cycle in the complex analytic setting. 

After defining our notion of residue, we construct the contraction over $\Spec(A)$ via generalized residues with values in line bundles over $\Spec(A)$ coming from tropical data, denoted $\res^m$, as follows. There are $n$ rational curves $R_1,...,R_n$ connecting the subcurve $E$ to the rest of $C$. Let $p_i$ be the node connecting $R_i$ to $E$. The condition for a function $f$ in a neighborhood of $E$ to descend to the contraction is
\[
\sum_{p_i} \res^m_{p_i,R_i}(f\phi) = 0,
\]
where $\phi$ is a specially chosen differential.

We then replace $S$ with the spectrum of a local noetherian ring $R$ with maximal ideal $\mm$ and assume $C$ is a smoothing. Setting $S_n = \Spec(R/\mm^{n+1})$ and $C_n = C \times_S S_n$, we have a diagram of cartesian squares
\begin{center}
    \begin{tikzcd}
C_0 \arrow[r] \arrow[d] & C_1 \arrow[r] \arrow[d] & C_2 \arrow[d] \arrow[r] & \cdots \\
S_0 \arrow[r]           & S_1 \arrow[r]           & S_2 \arrow[r]           & \cdots
\end{tikzcd}
\end{center}
and a contraction for each $i$, $\tau_i : C_i \to \bar C_i$. We then have the following characterization of the functions in $\O_{\widebar{C}}$. 

\begin{theorem}
    Let $U \subset \bar C_n$ be an open subset, $V = \tau_n^{-1}(U)$, and $f_n \in \O_{C_n}(V)$. The following are equivalent:
    \begin{enumerate}
        \item $f_n$ is in the image of $\O_{C_m}(V) \to \O_{C_n}(V)$ for all $m \geq n$;
        \item there exist compatible lifts $f_m \in \O_{C_m}(V)$ of $f_n$ for all $m \geq n$;
        \item $f_n$ is in $\O_{\bar C_n}(U)$;
        \item $f_n$ satisfies the residue condition: $\sum_{p_i} \res^m_{p_i,R_i}(f_n\phi_n|_V) = 0 \in \O_S(-\delta)$.
    \end{enumerate}
\end{theorem}

Here the $\phi_n$ are specially chosen differentials on the $C_n$. 

\subsection{Relation to Other Works and Future Directions}

Contractions of genus one curves have been studied previously, for example, by Smyth in \cite{smyth_mstable}, Ranganathan, Santos--Parker, and Wise in \cite{rspw}, and Bozlee in \cite{boz}. Comparing approaches, the constructions in \cite{smyth_mstable} and \cite{rspw} both rely on putting the curve in a smoothing family and performing the contraction in that family, and therefore struggle to identify the contraction for a fixed curve. The construction in \cite{boz} remedies this by constructing the ring for the contraction by hand, but the viewpoint therein does not make clear the connections between the contraction and other related problems, such as smoothing of differentials. 

Centrally aligned curves were introduced in \cite{rspw} in order to give a modular desingularization of $\widebar{\mathcal{M}}_{1,n}(\PP^r,d)$. By giving a more explicit construction of the contraction used in that paper, it makes clear why factorization through the contraction is related to smoothing stable maps, in terms of smoothing functions pulled back from projective space. We should also be able to use our viewpoint on the contraction to give a novel proof of the splitting of the virtual class of $\widebar{\mathcal{M}}_{1,n}(\PP^r,d)$ over the main and boundary components. 

The technique for the contraction constructed within should also generalize to higher genera by considering residues of functions paired with multiple independent differentials on a curve. Such contractions are vital to constructing modular desingularizations of the moduli spaces of stable maps to projective space, see, e.g., the aforementioned construction in \cite{rspw} and the similar construction of Battistella and Carocci in \cite{BC}. Extending this approach even to just genus two presents interesting challenges and questions, such as the presence of non-reduced structures on the contraction seen in \cite{BC}, coming from degenerating curves with a double cover of $\PP^1$. 

Lastly, by making the connection between the contraction and residues of differentials explicit, we hope to be able to give a new proof of the differential smoothing conjecture (valid in any characteristic) posed by Ranganathan and Wise. This conjecture was stated by Battistella and Bozlee in \cite{BB} and proven in characteristic 0 by Chen and Chen in \cite{CC}. In the process, we should also be able to give an algebraic proof of the global residue condition for smoothing of differentials introduced in \cite{BCGGM}.

\subsection{Acknowledgements}

First, I would like to thank my advisor, Jonathan Wise, for pushing me in the direction of this paper and offering endless insights and guidance. I would also like to thank Luca Battistella, Francesca Carocci, Dawei Chen, and Dhruv Ranganathan for helpful conversations, encouragement, and general interest in the content of this paper. A special thanks to Sebastian Bozlee for providing helpful comments on an earlier draft of this paper and for suggesting a fix for the proof of Lemma \ref{invert}. This research was partially conducted during the time that A.N. was supported by the National Science Foundation under Grant No. DMS-2231565.

\section{Background}

\subsection{Genus One Singularities}

We work over an algebraically closed field $k$ throughout.

Let $C$ be a reduced curve over $k$ and let $\nu : \tilde{C} \to C$ be the normalization at a closed point $p \in C$. Let $\{p_1,...,p_n\}$ be the preimages of $p$ under the map $\nu$. We define two quantities:
\begin{enumerate}
    \item the \textbf{number of branches of the singularity at $p$} is $m(p) = n$; 
    \item the \textbf{$\delta$-invariant} is $\delta(p) = \text{dim}_k(\nu_*\O_{\tilde{C}}/\O_C)$.
\end{enumerate}
We then define the \textbf{genus of the singularity at $p$} to be 
\[
g(p) = \delta(p) - m(p) + 1.
\]
We will be concerned with Gorenstein genus one singularities, which have been characterized by Smyth (\cite{smyth_mstable} Proposition A.3) as follows.

\begin{proposition}
    For each $n \geq 1$, there is a unique Gorenstein genus one singularity with $n$ branches. Specifically, these singularities are the following:
    \begin{enumerate}
        \item $n = 1$: the cusp $V(y^2-x^3)$;
        \item $n = 2$: the tacnode $V(y^2-yx^2)$;
        \item $n \geq 3$: the union of $n$ general lines through the origin in $\A^{n-1}$. 
    \end{enumerate}
\end{proposition}

\subsection{Tropical Curves}

We refer the reader to \cite{ccuw} for the details of the following approach to tropical curves. 

Let $M$ be a finitely generated, integral, saturated, sharp monoid. We define a tropical curve with edge lengths in $M$.

\begin{definition}
    An \textbf{$n$-marked tropical curve $\Gamma$ with edge lengths in $M$} is a finite graph with vertex set $V$, edge set $E$, and half-edge set $L$ along with:
    \begin{enumerate}
        \item a bijection $\{1,...,n\} \to L$;
        \item a genus function $g : V \to \N$;
        \item a length function $l : E \to M$.
    \end{enumerate}
\end{definition}

The genus of $\Gamma$ is the sum of the genera of each vertex plus the first Betti number of the underlying graph.

We will need the notion of a piecewise linear function on a tropical curve.

\begin{definition}
    A \textbf{piecewise linear function} $f$ on a tropical curve $\Gamma$ is the assignment of a value $f(v) \in M$ for each $v \in V$ and a slope $m(h) \in \N$ for each $h \in L$ such that if $e \in E$ is an edge between vertices $v,w \in V$, then $f(v) - f(w)$ is an integer multiple of $l(e)$, i.e., $f(v) - f(w) = sl(e)$, with $s \in \Z$.
\end{definition}

\subsection{Log Curves}

We recall the definition of a log scheme here, but we refer the reader to \cite{kato_log_structures} for more details and background on logarithmic geometry.

\begin{definition}
    Let $S$ be a scheme and $M_S$ be a (\'{e}tale) sheaf of monoids on $S$. A \textbf{logarithmic structure} on $S$ is morphism of sheaves of monoids $\epsilon : M_S \to \O_S$, where $\O_S$ is considered as a monoid under multiplication, such that $\epsilon^{-1}(\O_S^*) \to \O_S^*$ is an isomorphism. A \textbf{logarithmic scheme} is a scheme with a logarithmic structure. The sheaf $\widebar{M}_S = M_S/\O_S^*$ is the \textbf{characteristic monoid}. 
\end{definition}

We will not need to work with general log schemes, however, as we will be restricting our attention to log curves.

\begin{definition}
    A \textbf{log curve} over an fs log scheme $S$ is a log smooth, integral, and proper map $\pi : C \to S$ of fs log schemes such that each geometric fiber is a reduced and connected curve. 
\end{definition}

A more useful characterization is the following, due to F. Kato (\cite{fkato_deformations} 1.8). This formulation of the statement is borrowed from \cite{rspw} (Theorem 2.3.1).

\begin{theorem}
    Let $C \to S$ be a log curve and $p \in C$ a geometric point with image $s \in S$. Then there are \'{e}tale neighborhoods $V$ of $p$ and $U$ of $s$ so that $V \to U$ has a strict \`{e}tale map to a model $V' \to U$ with one of the following forms:
    \begin{enumerate}
        \item (smooth germ) $V' = \A^1_{U} \to U$, and the log structure on $V'$ is pulled back from the base;
        \item (germ of a marked point) $V' = \A^1_{U} \to U$, and the log structure on $V'$ is pulled back from the toric log structure on $\A^1$;
        \item (node) $V' = \Spec\, \O_{U}[x,y]/(xy-t)$ for some $t \in \O_{U}$, and the log structure on $V'$ is pulled back from the multiplication map $\A^2 \to \A^1$ of toric varieties along the map $U \to \A^1$ of log schemes induced by $t$.
    \end{enumerate}

    We call the image of $t \in M_S$ in $\widebar{M}_S$ a \textbf{smoothing parameter of the node}.
\end{theorem}

Using this characterization, we can define the tropical curve $\Gamma$ associated to a log curve $C \to S$ by defining the underlying graph of $\Gamma$ to be the dual graph of $C$, and the length of an edge $e \in E$ is the smoothing parameter $t \in \widebar{M}_S$ associated to the node, i.e., $l(e) = t$. 

We will also need the notion of the line bundle associated to a piecewise linear function, which we now review. Let $f$ be a piecewise linear function on the tropicalization of $C \to S$. It is shown in \cite{ccuw} (Remark 7.3) that one can view this as a section of $\widebar{M}_C$. Then from the short exact sequence
\[
0 \to \O_C^* \to M_C^{\text{gp}} \to \widebar{M}_C^{\text{gp}} \to 0
\]
we get a map $H^0(C, \widebar{M}_C^{\text{gp}}) \to H^1(C, \O_C^*)$. This is the association of a line bundle $\O(-f)$ to a piecewise linear function $f$, where $f$ is viewed as a section of $\widebar{M}_C^{\text{gp}}$. The following proposition from \cite{rspw} (Proposition 2.4.1) will be used throughout.

\begin{proposition}\label{pw}
    Let $\pi : C \to S$ be a log curve such that $\widebar{M}_S$ and $\Gamma$ are constant over $S$. If $f$ is a piecewise linear function on $\Gamma$ and $v$ is a vertex corresponding to an irreducible component $C_v$ of $C$, then we have
    \[
    \O_C(f)|_{C_v} = \O_{C_v}\left(\sum s(f,e_q) q\right) \otimes \pi^*\O_S(f(v)),
    \]
    where $s(f,e_q)$ is the outgoing slope of $f$ along the edge corresponding to $q$. 
\end{proposition}

\subsection{Centrally Aligned Curves}

We can now define the notion of a central alignment, as introduced in \cite{rspw} (Definition 4.6.2.1). Let $\pi : C \to S$ be a genus one log curve such that $S = \Spec(k)$ and let $\Gamma$ be the tropicalization of $C$. 

\begin{definition}
    The \textbf{core} of $C$ is the minimal genus one subcurve with respect to inclusion.
\end{definition}

We can now define a piecewise linear function $\lambda$ as follows. For any vertex $v$, there is a unique path from the core to $v$, say $e_1,e_2,...,e_k$. We then set 
\[
\lambda(v) = \sum_{i=1}^k l(e_i) \in \widebar{M}_S.
\]
It is shown in \cite{rspw} (in the paragraphs following Lemma 3.3.2) that this can be extended to have $S$ be an arbitrary base.

Now we can put a partial order on $\widebar{M}_S^{gp}$ by saying $f \leq g$ if $g-f \in \widebar{M}_S$. With this we can make the following definition:

\begin{definition}
    A genus one log curve $\pi : C \to S$ is \textbf{radially aligned} if for every geometric point $s \in S$ and every pair of vertices $v, w$ in $\Gamma_s$, $\lambda(v)$ and $\lambda(w)$ are comparable under the partial order defined above.
\end{definition}

We can now define the related notion of a (stable) central alignment. 

\begin{definition}
    Let $\pi : C \to S$ be a genus one log curve with tropicalization $\Gamma$. A \textbf{(stable) central alignment} of $C$ is the existence of $v$ in $\Gamma$ such that, setting $\delta = \lambda(v)$, the following two conditions hold:
    \begin{enumerate}
        \item the interior of the circle of radius $\delta$ around the core is radially aligned;
        \item the subgraph of $\Gamma$ where $\lambda < \delta$ is a stable curve.
    \end{enumerate}
    A curve with a central alignment is a \textbf{centrally aligned curve}.
\end{definition}

Given a central alignment, there is a semistable modification $\Tilde{C} \to C$ obtained by the following procedure. For every vertex $v$ and edge $e$ between vertices $w$ and $u$ contained in the aligned circle, if $\lambda(w) < \lambda(v) < \lambda(u)$, then we subdivide the edge $e$ to introduce a new vertex a distance of $\lambda(v)$ away from the core. Throughout the remainder of this paper, given a centrally aligned curve $C$, we will implicitly work on $\Tilde{C}$. 

\begin{example}
    Suppose we have two edges $e_1$ and $e_2$ connecting vertices $v_1$ and $v_2$ to the core, with $\lambda(v_1) < \lambda(v_2)$. We then subdivide the edge $e_2$ at a distance of $\lambda(v_1)$ from the core. See Figure \ref{fig:semistable}.

    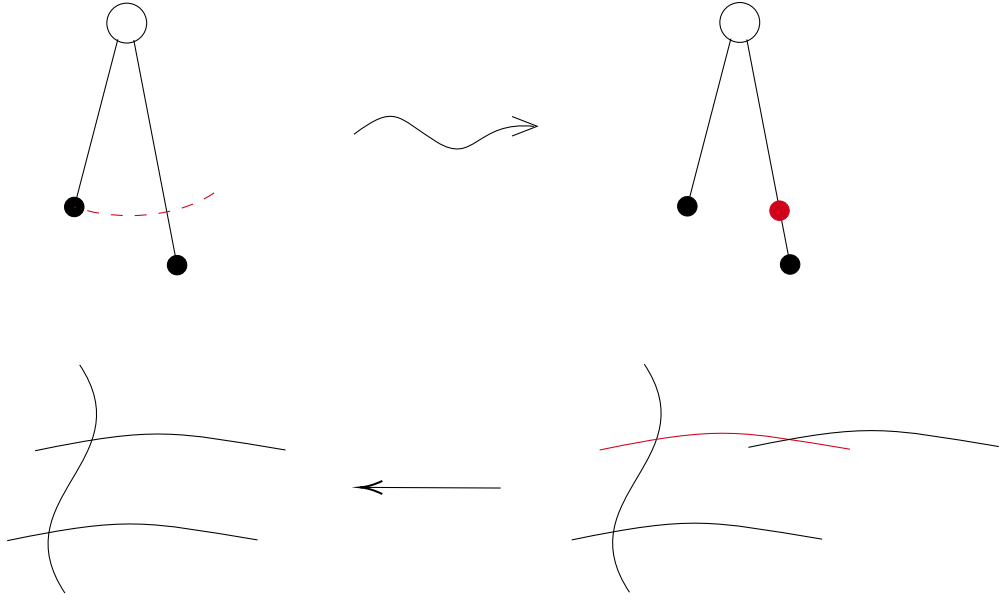
\begin{figure}
        \centering
        \begin{tikzpicture}[x=0.75pt,y=0.75pt,yscale=-1,xscale=1]

\draw   (148.67,29) .. controls (148.67,23.48) and (153.14,19) .. (158.67,19) .. controls (164.19,19) and (168.67,23.48) .. (168.67,29) .. controls (168.67,34.52) and (164.19,39) .. (158.67,39) .. controls (153.14,39) and (148.67,34.52) .. (148.67,29) -- cycle ;
\draw    (154.2,37.14) -- (132.33,121.27) ;
\draw  [fill={rgb, 255:red, 0; green, 0; blue, 0 }  ,fill opacity=1 ] (127.46,121.27) .. controls (127.46,118.58) and (129.64,116.4) .. (132.33,116.4) .. controls (135.02,116.4) and (137.2,118.58) .. (137.2,121.27) .. controls (137.2,123.96) and (135.02,126.14) .. (132.33,126.14) .. controls (129.64,126.14) and (127.46,123.96) .. (127.46,121.27) -- cycle ;
\draw  [fill={rgb, 255:red, 0; green, 0; blue, 0 }  ,fill opacity=1 ] (179.06,150.47) .. controls (179.06,147.78) and (181.24,145.6) .. (183.93,145.6) .. controls (186.62,145.6) and (188.8,147.78) .. (188.8,150.47) .. controls (188.8,153.16) and (186.62,155.34) .. (183.93,155.34) .. controls (181.24,155.34) and (179.06,153.16) .. (179.06,150.47) -- cycle ;
\draw    (162.2,37.54) -- (183.93,150.47) ;
\draw  [draw opacity=0][dash pattern={on 4.5pt off 4.5pt}] (202.51,114.2) .. controls (192.78,120.92) and (178.09,125.33) .. (161.59,125.61) .. controls (150.83,125.78) and (140.78,124.18) .. (132.33,121.27) -- (161.09,95.61) -- cycle ; \draw  [color={rgb, 255:red, 208; green, 2; blue, 27 }  ,draw opacity=1 ][dash pattern={on 4.5pt off 4.5pt}] (202.51,114.2) .. controls (192.78,120.92) and (178.09,125.33) .. (161.59,125.61) .. controls (150.83,125.78) and (140.78,124.18) .. (132.33,121.27) ;  

\draw   (456.33,28.67) .. controls (456.33,23.14) and (460.81,18.67) .. (466.33,18.67) .. controls (471.86,18.67) and (476.33,23.14) .. (476.33,28.67) .. controls (476.33,34.19) and (471.86,38.67) .. (466.33,38.67) .. controls (460.81,38.67) and (456.33,34.19) .. (456.33,28.67) -- cycle ;
\draw    (461.87,36.8) -- (440,120.93) ;
\draw  [fill={rgb, 255:red, 0; green, 0; blue, 0 }  ,fill opacity=1 ] (435.13,120.93) .. controls (435.13,118.24) and (437.31,116.06) .. (440,116.06) .. controls (442.69,116.06) and (444.87,118.24) .. (444.87,120.93) .. controls (444.87,123.62) and (442.69,125.8) .. (440,125.8) .. controls (437.31,125.8) and (435.13,123.62) .. (435.13,120.93) -- cycle ;
\draw  [fill={rgb, 255:red, 0; green, 0; blue, 0 }  ,fill opacity=1 ] (486.73,150.13) .. controls (486.73,147.44) and (488.91,145.26) .. (491.6,145.26) .. controls (494.29,145.26) and (496.47,147.44) .. (496.47,150.13) .. controls (496.47,152.82) and (494.29,155) .. (491.6,155) .. controls (488.91,155) and (486.73,152.82) .. (486.73,150.13) -- cycle ;
\draw    (469.87,37.2) -- (491.6,150.13) ;
\draw  [color={rgb, 255:red, 208; green, 2; blue, 27 }  ,draw opacity=1 ][fill={rgb, 255:red, 208; green, 2; blue, 27 }  ,fill opacity=1 ] (481.4,123.13) .. controls (481.4,120.44) and (483.58,118.26) .. (486.27,118.26) .. controls (488.96,118.26) and (491.14,120.44) .. (491.14,123.13) .. controls (491.14,125.82) and (488.96,128) .. (486.27,128) .. controls (483.58,128) and (481.4,125.82) .. (481.4,123.13) -- cycle ;
\draw    (272.67,84.83) .. controls (293.61,69.12) and (294.33,76.17) .. (314,88.5) .. controls (333.67,100.83) and (333.67,77.83) .. (363,80.83) ;
\draw   (352.05,75.87) -- (364.65,80.77) -- (352.05,85.67) ;

\draw    (135.05,200.52) .. controls (167.04,248.51) and (95.66,267.04) .. (127.64,315.04) ;
\draw    (112.67,243.62) .. controls (173.75,231) and (182.31,234.06) .. (238.33,243.24) ;
\draw    (98.67,288.76) .. controls (159.75,276.14) and (168.31,279.2) .. (224.33,288.38) ;
\draw    (418.38,200.13) .. controls (450.37,248.13) and (378.99,266.66) .. (410.98,314.65) ;
\draw [color={rgb, 255:red, 208; green, 2; blue, 27 }  ,draw opacity=1 ]   (396.01,243.24) .. controls (457.09,230.62) and (465.64,233.68) .. (521.67,242.86) ;
\draw    (382,288.37) .. controls (443.08,275.75) and (451.64,278.81) .. (507.66,287.99) ;
\draw    (470.67,241.91) .. controls (531.75,229.28) and (540.31,232.34) .. (596.33,241.52) ;
\draw    (346.33,262.33) -- (278,262.01) ;
\draw [shift={(276,262)}, rotate = 0.27] [color={rgb, 255:red, 0; green, 0; blue, 0 }  ][line width=0.75]    (10.93,-3.29) .. controls (6.95,-1.4) and (3.31,-0.3) .. (0,0) .. controls (3.31,0.3) and (6.95,1.4) .. (10.93,3.29)   ;

\end{tikzpicture}

        \caption{An example of the semistable modification. Here, one vertex is added to subdivide an edge, which corresponds to adding a rational component to the curve.}
        \label{fig:semistable}
    \end{figure}
\end{example}

\section{Residues}

Let $A$ be a local artinian ring with maximal ideal $\mm$ and residue field $k$, $\pi : C \to S = \Spec(A)$ a flat family of curves with reduced fiber, and $p$ a point in the fiber $C_0 = C \times_{\Spec(A)} \Spec(A/\mm)$. Given a minimal prime ideal $q_0 \subset \hat\O_{C_0,p}$, let $q \subset \hat\O_{C,p}$ be the preimage of $q_0$ under the map $\hat\O_{C,p} \to \hat\O_{C_0,p}$.

\begin{proposition}\label{local_ring}
    If $q_0$ is a minimal prime ideal of $\hat\O_{C_0,p}$, then there is an isomorphism $(\hat\O_{C_0,p})_{q_0} \simeq k(\!(x)\!)$. 
\end{proposition}

\begin{proof}
    Denote $\hat\O_{C_0,p}$ and $\O_{C_0,p}$ by $\hat{R}$ and $R$, respectively, throughout.
    
    Let $\{q_1,...,q_n\}$ be the minimal prime ideals of $\hat{R}$. If we denote by $R'$ the integral closure of $R$ in $Q(R)$ (its total ring of fractions), then the number of maximal ideals of $R'$ is $n$ (\cite{stacks-project} Tags 0C37 and 0C2E). Denote these by $\{m_1,...,m_n\}$. Let $\hat{R'}$ be the completion of $R'$ at the ideal $I = m_1\cdots m_n$, which is equivalent to completing $R'$ at the maximal ideal $m$ of $R$. Since $R'$ is semi-local, we have an isomorphism (\cite{Matsumura} Thm. 8.15)
    \[
    \hat{R'} \simeq \hat{R'}_{m_1} \times \cdots \times \hat{R'}_{m_n},
    \]
    where $\hat{R'}_{m_i}$ is the completion of the local ring $R'_{m_i}$ at its maximal ideal. 

    Now since $R'$ is finite over $R$, we also have an isomorphism
    \[
    \hat{R'} \simeq \hat{R} \otimes_R R'.
    \]
    $R$ is a localization of a finitely generated algebra over a field, so it is an excellent ring. In particular, it is a G-ring, so it has normal formal fibers, and from (\cite{egaiv} Prop. 6.14.4) we know that $\hat{R} \otimes_R R'$ is the integral closure of $\hat{R}$ in $Q(\hat{R})$. Thus, we can write
    \[
    \hat{R'} \simeq \hat{R}_1 \times \cdots \times \hat{R}_n,
    \]
    where $\hat{R}_i$ is the integral closure of $\hat{R}/q_i$ in $Q(\hat{R}/q_i) \simeq \hat{R}_{q_i}$ (\cite{stacks-project} Tag 030C). Therefore, up to reordering, we can assume $\hat{R}_i \simeq \hat{R'}_{m_i}$. In particular, $\hat{R}_i$ is a regular, complete, local noetherian ring of dimension one over $k$, so by the Cohen Structure Theorem (\cite{stacks-project} Tag 0C0S) $\hat{R}_i \simeq k[\![x]\!]$. Thus $\hat{R}_{q_i} \simeq Q(\hat{R}/q_i) \simeq k(\!(x)\!)$ for each $i$. 
\end{proof}
    
\begin{proposition}\label{stalk}
    If $q \subset \hat\O_{C,p}$ is the preimage of a minimal prime ideal $q_0 \subset \hat\O_{C_0,p}$, then $(\hat\O_{C})_q$ is isomorphic to $A(\!(x)\!)$.
\end{proposition}

\begin{proof}
    From Prop. \ref{local_ring}, we know that $(\hat\O_{C_0,p})_{q_0} \simeq k(\!(x)\!)$. For ease of notation, we denote $\hat\O_{C_0,p}$ and $\hat\O_{C,p}$ by $\hat\O_{C_0}$ and $\hat\O_{C}$, respectively, throughout.
    
    We now have two compatible infinitesimal deformations of $(\hat\O_{C_0})_{q_0}$ fitting in the following diagram

\begin{center}
    \begin{tikzcd}
{k(\!(x)\!) = (\hat\O_{C_0})_{q_0}} & {(\hat\O_{C})_{q}} \arrow[l] \\
A(\!(x)\!) \arrow[u]         & A \arrow[u] \arrow[l]         
\end{tikzcd}
\end{center}
We also have a compatible map $A(\!(x)\!) \to (\hat\O_{C})_q$ constructed as follows. The map $\hat{\O}_C \to \hat{\O}_{C_0}$ is surjective, so $(\hat{\O}_C)_q \to (\hat{\O}_{C_0})_{q_0} \simeq k(\!(x)\!)$ is also surjective. Let $y \in (\hat{\O}_C)_q$ be a preimage of $x$ under this map, which will be a unit since this is a local homomorphism. We then define $A(\!(x)\!) \to (\hat{\O}_C)_q$ by sending $x$ to $y$. 

We now show that the map just constructed $A(\!(x)\!) \to (\hat{\O}_C)_q$ is an isomorphism. Since $A$ is an order $N$ extension of $k$, $A(\!(x)\!)$ and $(\hat\O_{C})_q$ are both order $N$ extensions of $k(\!(x)\!)$. We show the desired isomorphism via induction on the order of extension $n$ by factoring these extensions into extensions of lower orders. We set $B_n$ and $D_n$ to be the order $n$ extensions of $k(\!(x)\!)$ coming from $A(\!(x)\!)$ and $(\hat\O_{C})_q$, respectively, $A_n$ to be the corresponding extension of $k$, and assume we have a map $B_n \to D_n$ over $A_n$ as constructed in the previous paragraph.

For the case $n = 1$, we have $B_1 = D_1 = k(\!(x)\!)$. Now assume the statement holds for $n-1$. We note that if $I$ is the kernel of $A_n \to A_{n-1}$, then by flatness of the deformations over $A_n$ and the base case, we know that the kernel of $B_n \to B_{n-1}$, denoted $J$, is $J = I \otimes_{A_n} B_n = I \otimes_{A_1} B_1$. Similarly, the kernel of $D_n \to D_{n-1}$ is $I \otimes_{A_1} D_1 \simeq I \otimes_{A_1} B_1 = J$. We then have the following diagram of exact sequences

\begin{center}
    \begin{tikzcd}
0 \arrow[r] & J \arrow[r]           & B_n \arrow[r]           & B_{n-1} \arrow[r]           & 0 \\
0 \arrow[r] & J \arrow[r] \arrow[u] & D_n \arrow[r] \arrow[u] & D_{n-1} \arrow[r] \arrow[u] & 0
\end{tikzcd}
\end{center}
The left vertical map is an isomorphism, as discussed above, and the right vertical map is an isomorphism via the induction hypothesis. Therefore, the middle map is also an isomorphism, which concludes the induction.

Hence we conclude that $(\hat\O_{C})_q \simeq A(\!(x)\!)$.
\end{proof}

Now let $p \in C$ be a point, $q \subset \hat\O_{C,p}$ a prime as above (corresponding to a branch, $X$, at the point $p$ in the fiber $C_0$), and $\phi \in \omega_C(U)$ for some open set $U$ containing $p$. 

\begin{lemma}\label{differential}
    Assume that $C$ is Gorenstein. $(\hat\omega_{C,p})_q$ is isomorphic to $A(\!(x)\!)\,dx$.
\end{lemma}

\begin{proof}
    Since $\omega_C$ is a line bundle on $C$, the stalk at $p$ will be one-dimensional over the stalk of $\O_C$. The localization of the completion will still be one-dimensional over the localization of the completion of the stalk of $\O_C$, which we know is isomorphic to $A(\!(x)\!)$ by Lemma \ref{stalk}. Therefore, we need only show that $dx$ is a basis element.

    Let $\tilde{C_0}$ be the normalization of $C_0$, let $\tilde{X}$ be the component of $\tilde{C_0}$ corresponding to the branch $X$ of $C_0$, and let $\tilde{p} \in \tilde{X}$ be the point mapping to $p \in C_0$. Let $\tilde{X}'$ be a deformation of $\tilde{X}$ over $A$. If we let $C^{\circ} \subset C$ be the open subset where we have removed $p$ and all components other than $X$, then, since a proper curve minus a finite number of points is affine, and infinitesimal deformations of affines are trivial (see, e.g., \cite{hartdef} Cor. 4.8), we see that $C^{\circ}$ is the same as $\tilde{X}' \setminus \{\tilde{p}\}$. Thus we deduce that $(\O_{C,p})_q$ is isomorphic to $(\O_{\tilde{X}',\tilde{p}})_q$. Now since $\tilde{X}'$ is a smooth curve, we know that $\omega_{\tilde{X}',\tilde{p}}$ is one dimensional over $\O_{\tilde{X}',\tilde{p}}$ and generated by $d$ of a uniformizer of $\O_{\tilde{X}',\tilde{p}}$. This is still true after localization and completion. Now we have a map $(\hat\omega_{\tilde{X}',\tilde{p}})_q \to (\hat\omega_{C,p})_q$, and the generator of the former will map to $dx$, from which we conclude that $dx$ is a basis element. 
\end{proof}

Now from Lemma \ref{differential}, a differential $\phi \in \omega_C(U)$ will have $(\phi_p)|_X \in A(\!(x)\!) \, dx$. Let us define a map $A(\!(x)\!) \, dx \to A$ by sending $\sum a_ix^i \mapsto a_{-1}$. 

\begin{definition}\label{residue}
    The \textbf{residue of $\phi$ at $p$ along the branch $X$} is the image of $(\phi_p)|_X$ under the map $A(\!(x)\!) \, dx \to A$. We denote this by $\res_{p,X}(\phi)$. 

    This induces a map $\pi_*\omega_C \to \O_S$, which we call the \textbf{residue map}.
\end{definition}

We must now show that this is well-defined, i.e., that the residue of an element of $A(\!(x)\!)\, dx$ is invariant under any continuous automorphism of $A(\!(x)\!)$. This follows from \cite{LaurentHoms} Corollary 5.4. This statement is more powerful than needed, as it deals with Laurent series in $n$ variables. In order to keep this self-contained, we provide a proof in Appendix \ref{inv} of the case $n = 1$ that we are considering.

\subsection{Facts about Residues}

\begin{remark}
    We note that, by definition, if we have an expansion
\begin{align}\label{res_coordinates}
    (\phi_p)|_X = \sum a_ix^i \, dx \in A(\!(x)\!) \, dx,
\end{align}
then the residue is just the term $a_{-1}$ in \eqref{res_coordinates}. We will use this fact heavily throughout.
\end{remark}

\begin{remark}
    When $p$ is a smooth point, there is only a single branch and we recover the usual definition of residue on a smooth curve over a local artinian ring (see, e.g., Chapter VII of \cite{hartRD}).
\end{remark}

As mentioned in the Introduction, one may think of this residue over a local artinian ring as an algebraic analogue of integration around a vanishing cycle in the complex analytic setting. The following lemma lends credence to that idea.

\begin{lemma}\label{invert}
    If $p \in C$ is a node with branches $X$ and $Y$, then 
    \[
    \res_{p,X}(\phi) = - \res_{p,Y}(\phi).
    \]
\end{lemma}

\begin{proof}
    Assume we have local coordinates $x$ and $y$ for $X$ and $Y$, respectively, such that $xy = t$ for some $t \in \mm$. We first write out
    \[
    \phi_p|_X = \left(\sum_{j=N}^{\infty} a_jx^j\right)dx,
    \]
    so $\res_{p,X}(\phi) = a_{-1}$. We now split into the cases $t = 0$ and $t \neq 0$.

    First assume $xy = t$, where $t \neq 0$. When we restrict to the complement of $X$, $y$ will be invertible and we can write $x = ty^{-1}$, and $dx = -t\frac{dy}{y^2}$. The $-1$ term for the residue in this expansion on $Y$ will come from $a_{-1}x^{-1} = a_{-1}(ty^{-1})^{-1}$, due to the $1/y^2$ in the $dx$ conversion. Multiplying $-t\frac{dy}{y^2}$ by this term shows 
    \[
    \res_{p,Y}(\phi) = -a_{-1} = -\res_{p,X}(\phi).
    \]

    Now assume $t = 0$. Over $\Spec(k)$, we have that 
    \[
    \omega_{C/k} \subset \omega_{X \setminus p / k} \times \omega_{Y \setminus p / k} \simeq \Omega_{X \setminus p / k} \times \Omega_{Y \setminus p / k}
    \]
    and we can identify this with the multiples of $\left(\frac{dx}{x},-\frac{dy}{y}\right)$ (see, e.g., Rosenlicht duality in Altman--Kleinman \cite{altman-kleinman} Prop. VIII.1.16). Over an artinian base $S$ we still have
    \[
    \omega_{C/S} \subset \omega_{X \setminus p / S} \times \omega_{Y \setminus p / S} \simeq \Omega_{X \setminus p / S} \times \Omega_{Y \setminus p / S}
    \]
    and it will still be generated by one element. Using Nakayama's Lemma, one shows that $\left(\frac{dx}{x},-\frac{dy}{y}\right)$ still generates, since it generates over $k$. Thus, since every element of $\omega_{C/S}$ is a multiple of $\left(\frac{dx}{x},-\frac{dy}{y}\right)$ and no element of $\O_C$ will have different constant value on $X$ and $Y$ (as they agree at the node), we see that the residues of $\phi$ along $X$ and $Y$ still differ by a sign.
\end{proof}

\begin{lemma}\label{res_thm}
    If $C \to S$ is smooth, then 
    \[
    \sum_{p \in C} \res_{p,C}(\phi) = 0.
    \]
\end{lemma}

\begin{proof}
    See Theorem 1.5 and the remark following it in Chapter VII of \cite{hartRD}.
\end{proof}

\section{The Contraction}

\subsection{Construction}

Let $\pi : C \to S$ be a family of centrally aligned curves with $n$ rational curves emerging from the aligned circle. We will construct a contraction $\tau : C \to \widebar{C}$ that is an isomorphism outside of the aligned circle of $C$ and such that the interior of the aligned circle is contracted to a Gorenstein genus 1 singularity with $n$ branches in $\widebar{C}$. 

We first construct the contraction when $S$ is the spectrum of a local artinian ring $A$ with maximal ideal $\mm$ and residue field $k$. Set $C_0$ to be the scheme $C \times_S \Spec(A/\mm)$. Assume that the central alignment on $C_0$ splits it into $m$ layers $L_0,...,L_m$, where $L_0$ is the core, $L_i$ is the subcurve consisting of the union of the rational curves at level $i$, and $L_m$ is the union of the rational curves on the edge of the aligned circle. Since $C_0$ is centrally aligned, a smoothing parameter at any node connecting layer $L_i$ to $L_{i-1}$ will be the same, say equal to $t_i \in \mm$. We denote by $t$ the product of these smoothing parameters, $t = t_1 \cdots t_m$. See Figure \ref{fig:layers1}.

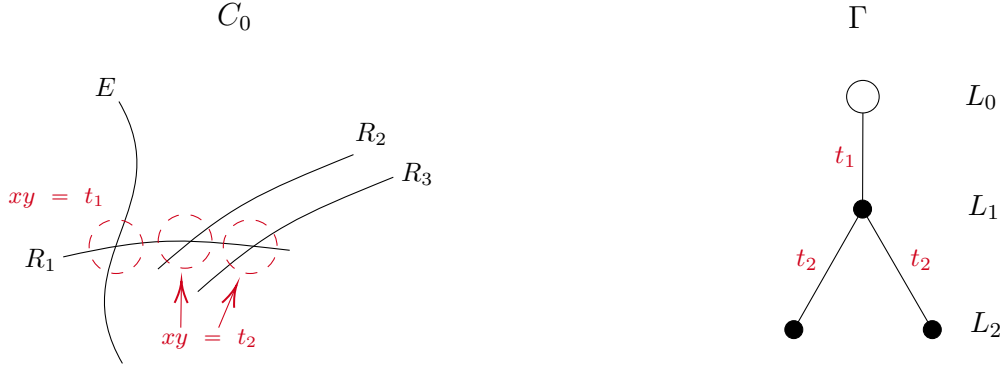
\begin{figure}
    \centering
\begin{tikzpicture}[x=0.75pt,y=0.75pt,yscale=-1,xscale=1]

\draw [color={rgb, 255:red, 0; green, 0; blue, 0 }  ,draw opacity=1 ]   (114.62,92.01) .. controls (145.17,144.18) and (85.83,171.27) .. (116.38,223.45) ;
\draw [color={rgb, 255:red, 0; green, 0; blue, 0 }  ,draw opacity=1 ]   (86.67,170) .. controls (133.67,159.33) and (148.33,160.67) .. (200.33,166.67) ;
\draw [color={rgb, 255:red, 0; green, 0; blue, 0 }  ,draw opacity=1 ]   (134.2,176.1) .. controls (170.33,144.2) and (183.84,138.34) .. (232.34,118.66) ;
\draw [color={rgb, 255:red, 0; green, 0; blue, 0 }  ,draw opacity=1 ]   (154.2,187.44) .. controls (190.33,155.54) and (203.84,149.67) .. (252.34,130) ;
\draw   (479.83,89.58) .. controls (479.83,85.07) and (483.49,81.42) .. (488,81.42) .. controls (492.51,81.42) and (496.17,85.07) .. (496.17,89.58) .. controls (496.17,94.09) and (492.51,97.75) .. (488,97.75) .. controls (483.49,97.75) and (479.83,94.09) .. (479.83,89.58) -- cycle ;
\draw  [fill={rgb, 255:red, 0; green, 0; blue, 0 }  ,fill opacity=1 ] (483.33,146) .. controls (483.33,143.51) and (485.35,141.5) .. (487.83,141.5) .. controls (490.32,141.5) and (492.33,143.51) .. (492.33,146) .. controls (492.33,148.49) and (490.32,150.5) .. (487.83,150.5) .. controls (485.35,150.5) and (483.33,148.49) .. (483.33,146) -- cycle ;
\draw  [fill={rgb, 255:red, 0; green, 0; blue, 0 }  ,fill opacity=1 ] (518.33,206.5) .. controls (518.33,204.01) and (520.35,202) .. (522.83,202) .. controls (525.32,202) and (527.33,204.01) .. (527.33,206.5) .. controls (527.33,208.99) and (525.32,211) .. (522.83,211) .. controls (520.35,211) and (518.33,208.99) .. (518.33,206.5) -- cycle ;
\draw  [fill={rgb, 255:red, 0; green, 0; blue, 0 }  ,fill opacity=1 ] (448.83,206.5) .. controls (448.83,204.01) and (450.85,202) .. (453.33,202) .. controls (455.82,202) and (457.83,204.01) .. (457.83,206.5) .. controls (457.83,208.99) and (455.82,211) .. (453.33,211) .. controls (450.85,211) and (448.83,208.99) .. (448.83,206.5) -- cycle ;
\draw    (488,97.75) -- (487.83,146) ;
\draw    (487.83,146) -- (453.33,206.5) ;
\draw    (487.83,146) -- (522.83,206.5) ;
\draw  [color={rgb, 255:red, 208; green, 2; blue, 27 }  ,draw opacity=1 ][dash pattern={on 4.5pt off 4.5pt}] (167,164.83) .. controls (167,157.38) and (173.04,151.33) .. (180.5,151.33) .. controls (187.96,151.33) and (194,157.38) .. (194,164.83) .. controls (194,172.29) and (187.96,178.33) .. (180.5,178.33) .. controls (173.04,178.33) and (167,172.29) .. (167,164.83) -- cycle ;
\draw  [color={rgb, 255:red, 208; green, 2; blue, 27 }  ,draw opacity=1 ][dash pattern={on 4.5pt off 4.5pt}] (99.67,164.83) .. controls (99.67,157.38) and (105.71,151.33) .. (113.17,151.33) .. controls (120.62,151.33) and (126.67,157.38) .. (126.67,164.83) .. controls (126.67,172.29) and (120.62,178.33) .. (113.17,178.33) .. controls (105.71,178.33) and (99.67,172.29) .. (99.67,164.83) -- cycle ;
\draw  [color={rgb, 255:red, 208; green, 2; blue, 27 }  ,draw opacity=1 ][dash pattern={on 4.5pt off 4.5pt}] (134.33,162.17) .. controls (134.33,154.71) and (140.38,148.67) .. (147.83,148.67) .. controls (155.29,148.67) and (161.33,154.71) .. (161.33,162.17) .. controls (161.33,169.62) and (155.29,175.67) .. (147.83,175.67) .. controls (140.38,175.67) and (134.33,169.62) .. (134.33,162.17) -- cycle ;
\draw [color={rgb, 255:red, 208; green, 2; blue, 27 }  ,draw opacity=1 ]   (145.67,205.5) -- (145.97,184.17) ;
\draw [shift={(146,182.17)}, rotate = 90.82] [color={rgb, 255:red, 208; green, 2; blue, 27 }  ,draw opacity=1 ][line width=0.75]    (10.93,-3.29) .. controls (6.95,-1.4) and (3.31,-0.3) .. (0,0) .. controls (3.31,0.3) and (6.95,1.4) .. (10.93,3.29)   ;
\draw [color={rgb, 255:red, 208; green, 2; blue, 27 }  ,draw opacity=1 ]   (164.67,205.5) -- (172.89,186.01) ;
\draw [shift={(173.67,184.17)}, rotate = 112.87] [color={rgb, 255:red, 208; green, 2; blue, 27 }  ,draw opacity=1 ][line width=0.75]    (10.93,-3.29) .. controls (6.95,-1.4) and (3.31,-0.3) .. (0,0) .. controls (3.31,0.3) and (6.95,1.4) .. (10.93,3.29)   ;

\draw (162.67,41.33) node [anchor=north west][inner sep=0.75pt]   [align=left] {$\displaystyle C_{0}$};
\draw (479.33,43) node [anchor=north west][inner sep=0.75pt]   [align=left] {$\displaystyle \Gamma $};
\draw (538,82.33) node [anchor=north west][inner sep=0.75pt]  [font=\small] [align=left] {$\displaystyle L_{0}$};
\draw (540.5,196.67) node [anchor=north west][inner sep=0.75pt]  [font=\small] [align=left] {$\displaystyle L_{2}$};
\draw (539.33,137.5) node [anchor=north west][inner sep=0.75pt]  [font=\small] [align=left] {$\displaystyle L_{1}$};
\draw (101.33,79) node [anchor=north west][inner sep=0.75pt]  [font=\footnotesize] [align=left] {$\displaystyle E$};
\draw (66,164.67) node [anchor=north west][inner sep=0.75pt]  [font=\footnotesize] [align=left] {$\displaystyle R_{1}$};
\draw (255.01,121.66) node [anchor=north west][inner sep=0.75pt]  [font=\footnotesize] [align=left] {$\displaystyle R_{3}$};
\draw (232,101.33) node [anchor=north west][inner sep=0.75pt]  [font=\footnotesize] [align=left] {$\displaystyle R_{2}$};
\draw (472.67,113.33) node [anchor=north west][inner sep=0.75pt]  [font=\footnotesize,color={rgb, 255:red, 208; green, 2; blue, 27 }  ,opacity=1 ] [align=left] {$\displaystyle t_{1}$};
\draw (510.67,165.33) node [anchor=north west][inner sep=0.75pt]  [font=\footnotesize,color={rgb, 255:red, 208; green, 2; blue, 27 }  ,opacity=1 ] [align=left] {$\displaystyle t_{2}$};
\draw (453.33,165.33) node [anchor=north west][inner sep=0.75pt]  [font=\footnotesize,color={rgb, 255:red, 208; green, 2; blue, 27 }  ,opacity=1 ] [align=left] {$\displaystyle t_{2}$};
\draw (58,134.33) node [anchor=north west][inner sep=0.75pt]  [font=\scriptsize,color={rgb, 255:red, 208; green, 2; blue, 27 }  ,opacity=1 ] [align=left] {$\displaystyle xy\ =\ t_{1}$};
\draw (134.67,205) node [anchor=north west][inner sep=0.75pt]  [font=\scriptsize,color={rgb, 255:red, 208; green, 2; blue, 27 }  ,opacity=1 ] [align=left] {$\displaystyle xy\ =\ t_{2}$};

\end{tikzpicture}

    \caption{A curve with the associated layers coming from the tropicalization $\Gamma$. Here $L_0$ consists of $E$, $L_1$ consists of $R_1$, and $L_2$ consists of $R_2$ and $R_3$. The smoothing parameters of the nodes connecting the layers are labeled on the associated edges in $\Gamma$. }
    \label{fig:layers1}
\end{figure}

We now define a refinement of residues that take values in twists of $\O_S$ by values of $-\lambda$. Let $R_1,...,R_k$ be the rational curves in layer $L_i$, and let $P_1,...,P_k$ be the nodes connecting $L_i$ to $L_{i-1}$. For any of these $R_j$, if $v$ is its corresponding vertex in the dual graph, then set $\lambda(v) = u$. On each $R_j \setminus \{P_j, Q_1,...,Q_l\}$, where $Q_1,...,Q_l$ are the other nodes on $R_j$, we have $\omega_C(-\lambda)|_{R_j \setminus \{P_j,Q_1,...,Q_l\}} = \omega_{R_j}|_{R_j \setminus \{P_j,Q_1,...,Q_l\}} \otimes \pi^*\O_S(-u)$. Applying the residue map $\pi_*\omega_C \to \O_S$, we find a residue in $\O_S(-u)$. For a differential $\phi \in H^0(C,\omega_C(-\lambda))$, we denote this by $\res^i_{P_j,R_j}(\phi) \in \O_S(-u)$. Elements of $\O_S(-u)$ are of the form $[t_1\cdots t_i]\O_S$, where we view $[t_1\cdots t_i]$ as a formal parameter. We then have a map $\O_S(-u) \to \O_S$ by sending $[t_1\cdots t_i]$ to the element $t_1\cdots t_i \in \O_S$, and we see that $\res_{P_j,R_j}(\phi)$ is the image of $\res^i_{P_j,R_j}(\phi)$ under this map. 

\begin{example}
    To illustrate the difference between the residues $\res^i$ and $\res$, we give an example where the base is the spectrum of $k$, $C \to S = \Spec(k)$. 

    At a node $P_j$ on the branch $R_j$, assume that $\phi$ has a local expansion
    \[
    \phi_{P_j}|_{R_j} = u\sum a_ix^i \, dx,
    \]
    where $u$ is the value of $\lambda$ at the corresponding vertex in the dual graph of $C$. In this case, we have
    \[
    \res^i_{P_j,R_j}(\phi) = [u](a_{-1}),
    \]
    but since the image of any of the smoothing parameters $t_i$ is 0 in $k$, we see that 
    \[
    \res_{P_j,R_j}(\phi) = 0.
    \]
\end{example}

\begin{lemma}\label{sections}
    The space $H^0(C,\omega_C(-\lambda))$ is free of rank one over $A$.
\end{lemma}

\begin{proof}
    We have an isomorphism $\O_C(\lambda) \simeq \omega_C$ (see \cite{rspw} Lemma 3.3.2.), so $\O_C(-\lambda) \simeq \omega_C^{\vee}$. Thus $\omega_C(-\lambda) \simeq \O_C$, and $H^0(C,\O_C)$ is free of rank one over $A$, which gives the result.
\end{proof}

We now construct the contraction. Let $E$ be the closed subcurve corresponding to the interior of the aligned circle in $C$, and let $\tau' : C \to C'$ be the topological contraction sending all of $E$ to a point $P$ in $C'$ and give this the structure sheaf $\O_{C'} = \tau'_*\O_C$. Note that, in the closed fiber, this space is a union of $n$ rational curves meeting transversally at a single point. See Figure \ref{fig:C'}. The underlying topological space of $\widebar{C}$ will be the same as $C'$ and the structure sheaf will be a subsheaf $\O_{\widebar{C}} \subset \tau'_*\O_C$ that we now describe.

\begin{figure}
    \centering
    \begin{tikzpicture}[x=0.75pt,y=0.75pt,yscale=-1,xscale=1]

\draw [color={rgb, 255:red, 208; green, 2; blue, 27 }  ,draw opacity=1 ]   (132.88,95.35) .. controls (168.92,140.64) and (98.57,164.55) .. (134.6,209.85) ;
\draw    (116.45,128.93) .. controls (171.07,116.72) and (186.79,119.77) .. (234.33,128.93) ;
\draw    (383.33,95.51) -- (459.67,169.81) ;
\draw    (451.63,94.66) -- (391.37,170.66) ;
\draw  [color={rgb, 255:red, 208; green, 2; blue, 27 }  ,draw opacity=1 ][fill={rgb, 255:red, 208; green, 2; blue, 27 }  ,fill opacity=1 ] (418.49,132.66) .. controls (418.49,131.07) and (419.84,129.78) .. (421.5,129.78) .. controls (423.16,129.78) and (424.51,131.07) .. (424.51,132.66) .. controls (424.51,134.25) and (423.16,135.54) .. (421.5,135.54) .. controls (419.84,135.54) and (418.49,134.25) .. (418.49,132.66) -- cycle ;
\draw    (262,142.33) -- (353.67,142.33) ;
\draw [shift={(355.67,142.33)}, rotate = 180] [color={rgb, 255:red, 0; green, 0; blue, 0 }  ][line width=0.75]    (10.93,-3.29) .. controls (6.95,-1.4) and (3.31,-0.3) .. (0,0) .. controls (3.31,0.3) and (6.95,1.4) .. (10.93,3.29)   ;
\draw    (107.79,176.93) .. controls (162.41,164.72) and (178.12,167.77) .. (225.67,176.93) ;

\draw (146,53.99) node [anchor=north west][inner sep=0.75pt]   [align=left] {$\displaystyle C$};
\draw (411.33,49.99) node [anchor=north west][inner sep=0.75pt]   [align=left] {$\displaystyle C'$};

\end{tikzpicture}

    \caption{The construction of $C'$.}
    \label{fig:C'}
\end{figure}
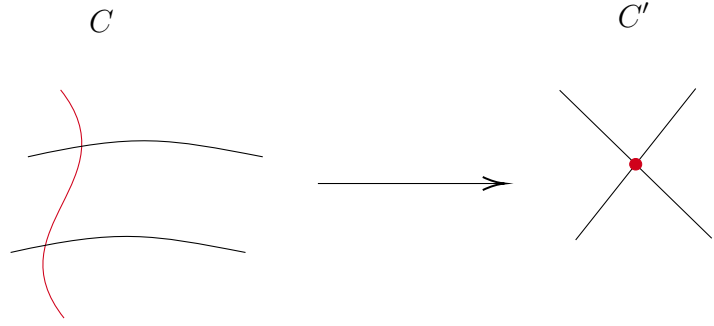

To define $\O_{\widebar{C}}$, we first note that $\tau$ will be an isomorphism outside of $E$, so we need only examine the case where $U \subset C'$ contains the point $P$. Note that $V = (\tau')^{-1}(U)$ is an open subset of $C$ containing $E$. Next, we fix a choice of generator $\phi \in H^0(C, \omega_C(-\lambda))$. Let $R_1,...,R_n$ be the rational curves in the outermost layer, $L_m$, and let $p_1,...,p_n$ be the nodes connecting $L_m$ to $L_{m-1}$. See Figure \ref{fig:layers}

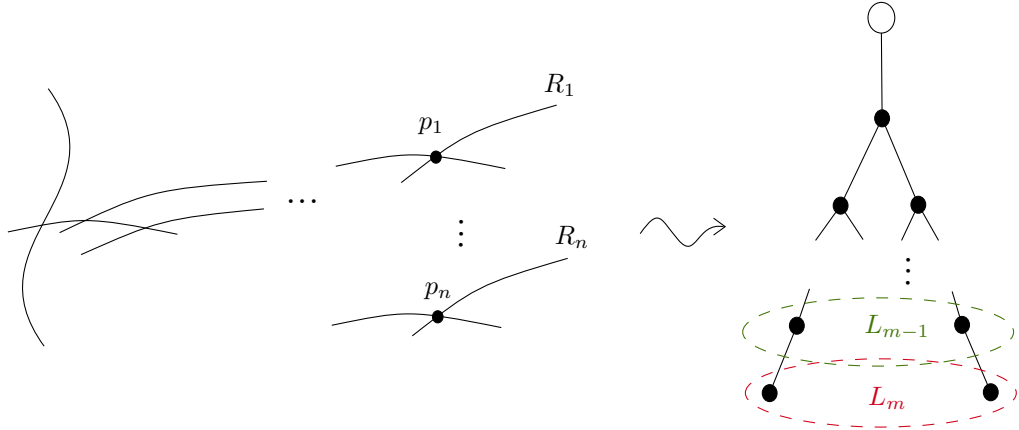
\begin{figure}
    \centering
\begin{tikzpicture}[x=0.75pt,y=0.75pt,yscale=-1,xscale=1,scale=0.9]

\draw    (35.84,73.59) .. controls (78.44,131.94) and (-9.05,158.86) .. (33.56,217.22) ;
\draw    (13.33,153.51) .. controls (54.01,145.16) and (57.1,144.52) .. (107.87,154.8) ;
\draw    (42.2,153.92) .. controls (89.86,131.13) and (93.47,129.39) .. (157.81,125.24) ;
\draw    (54.16,166.17) .. controls (96.48,147.39) and (99.69,145.95) .. (156.13,140.66) ;
\draw    (196.19,116.89) .. controls (236.87,108.54) and (239.96,107.9) .. (290.73,118.18) ;
\draw    (193.95,205.56) .. controls (234.62,197.2) and (237.71,196.56) .. (288.49,206.84) ;
\draw    (232.68,126.03) .. controls (266.79,99.31) and (269.37,97.27) .. (319.36,82.78) ;
\draw    (238.85,211.48) .. controls (272.96,184.76) and (275.54,182.72) .. (325.53,168.23) ;
\draw  [fill={rgb, 255:red, 0; green, 0; blue, 0 }  ,fill opacity=1 ] (249.1,111.79) .. controls (249.1,109.92) and (250.42,108.41) .. (252.05,108.41) .. controls (253.67,108.41) and (254.99,109.92) .. (254.99,111.79) .. controls (254.99,113.65) and (253.67,115.16) .. (252.05,115.16) .. controls (250.42,115.16) and (249.1,113.65) .. (249.1,111.79) -- cycle ;
\draw  [fill={rgb, 255:red, 0; green, 0; blue, 0 }  ,fill opacity=1 ] (250.11,200.83) .. controls (250.11,198.97) and (251.43,197.46) .. (253.06,197.46) .. controls (254.68,197.46) and (256,198.97) .. (256,200.83) .. controls (256,202.7) and (254.68,204.21) .. (253.06,204.21) .. controls (251.43,204.21) and (250.11,202.7) .. (250.11,200.83) -- cycle ;
\draw    (365.94,154.47) .. controls (376.73,139.33) and (377.11,146.11) .. (387.24,158) .. controls (397.38,169.89) and (397.38,147.72) .. (412.5,150.61) ;
\draw   (406.85,145.83) -- (413.35,150.55) -- (406.85,155.27) ;

\draw   (492.98,34.05) .. controls (492.98,29.33) and (496.32,25.5) .. (500.45,25.5) .. controls (504.57,25.5) and (507.92,29.33) .. (507.92,34.05) .. controls (507.92,38.78) and (504.57,42.61) .. (500.45,42.61) .. controls (496.32,42.61) and (492.98,38.78) .. (492.98,34.05) -- cycle ;
\draw  [fill={rgb, 255:red, 0; green, 0; blue, 0 }  ,fill opacity=1 ] (496.72,90.17) .. controls (496.72,87.56) and (498.57,85.44) .. (500.85,85.44) .. controls (503.13,85.44) and (504.97,87.56) .. (504.97,90.17) .. controls (504.97,92.77) and (503.13,94.89) .. (500.85,94.89) .. controls (498.57,94.89) and (496.72,92.77) .. (496.72,90.17) -- cycle ;
\draw  [fill={rgb, 255:red, 0; green, 0; blue, 0 }  ,fill opacity=1 ] (516.92,138.35) .. controls (516.92,135.74) and (518.77,133.63) .. (521.05,133.63) .. controls (523.32,133.63) and (525.17,135.74) .. (525.17,138.35) .. controls (525.17,140.96) and (523.32,143.07) .. (521.05,143.07) .. controls (518.77,143.07) and (516.92,140.96) .. (516.92,138.35) -- cycle ;
\draw  [fill={rgb, 255:red, 0; green, 0; blue, 0 }  ,fill opacity=1 ] (473.58,138.83) .. controls (473.58,136.23) and (475.43,134.11) .. (477.7,134.11) .. controls (479.98,134.11) and (481.83,136.23) .. (481.83,138.83) .. controls (481.83,141.44) and (479.98,143.56) .. (477.7,143.56) .. controls (475.43,143.56) and (473.58,141.44) .. (473.58,138.83) -- cycle ;
\draw  [fill={rgb, 255:red, 0; green, 0; blue, 0 }  ,fill opacity=1 ] (449.17,205.81) .. controls (449.17,203.2) and (451.02,201.09) .. (453.3,201.09) .. controls (455.57,201.09) and (457.42,203.2) .. (457.42,205.81) .. controls (457.42,208.42) and (455.57,210.53) .. (453.3,210.53) .. controls (451.02,210.53) and (449.17,208.42) .. (449.17,205.81) -- cycle ;
\draw  [fill={rgb, 255:red, 0; green, 0; blue, 0 }  ,fill opacity=1 ] (541.33,205.33) .. controls (541.33,202.72) and (543.18,200.61) .. (545.45,200.61) .. controls (547.73,200.61) and (549.58,202.72) .. (549.58,205.33) .. controls (549.58,207.94) and (547.73,210.05) .. (545.45,210.05) .. controls (543.18,210.05) and (541.33,207.94) .. (541.33,205.33) -- cycle ;
\draw  [fill={rgb, 255:red, 0; green, 0; blue, 0 }  ,fill opacity=1 ] (434.02,243.4) .. controls (434.02,240.79) and (435.87,238.67) .. (438.15,238.67) .. controls (440.43,238.67) and (442.27,240.79) .. (442.27,243.4) .. controls (442.27,246) and (440.43,248.12) .. (438.15,248.12) .. controls (435.87,248.12) and (434.02,246) .. (434.02,243.4) -- cycle ;
\draw  [fill={rgb, 255:red, 0; green, 0; blue, 0 }  ,fill opacity=1 ] (557.32,242.91) .. controls (557.32,240.31) and (559.17,238.19) .. (561.44,238.19) .. controls (563.72,238.19) and (565.57,240.31) .. (565.57,242.91) .. controls (565.57,245.52) and (563.72,247.64) .. (561.44,247.64) .. controls (559.17,247.64) and (557.32,245.52) .. (557.32,242.91) -- cycle ;
\draw    (500.45,42.61) -- (500.85,90.17) ;
\draw    (500.85,90.17) -- (477.7,138.83) ;
\draw    (500.85,90.17) -- (521.05,138.35) ;
\draw    (463.73,158.25) -- (477.7,138.83) ;
\draw    (490.24,157.77) -- (477.7,138.83) ;
\draw    (511.7,158.25) -- (521.05,138.35) ;
\draw    (521.05,138.35) -- (532.32,158.25) ;
\draw    (460.37,185.24) -- (453.3,205.81) ;
\draw    (539.9,186.68) -- (545.45,205.33) ;
\draw    (438.15,243.4) -- (453.3,205.81) ;
\draw    (561.44,242.91) -- (545.45,205.33) ;
\draw  [color={rgb, 255:red, 208; green, 2; blue, 27 }  ,draw opacity=1 ][dash pattern={on 4.5pt off 4.5pt}] (424.6,243.3) .. controls (424.6,232.79) and (458.42,224.27) .. (500.13,224.27) .. controls (541.85,224.27) and (575.67,232.79) .. (575.67,243.3) .. controls (575.67,253.81) and (541.85,262.33) .. (500.13,262.33) .. controls (458.42,262.33) and (424.6,253.81) .. (424.6,243.3) -- cycle ;
\draw  [color={rgb, 255:red, 65; green, 117; blue, 5 }  ,draw opacity=1 ][dash pattern={on 4.5pt off 4.5pt}] (423.11,209.3) .. controls (423.11,198.79) and (456.93,190.26) .. (498.64,190.26) .. controls (540.36,190.26) and (574.18,198.79) .. (574.18,209.3) .. controls (574.18,219.81) and (540.36,228.33) .. (498.64,228.33) .. controls (456.93,228.33) and (423.11,219.81) .. (423.11,209.3) -- cycle ;

\draw (166.6,133.77) node [anchor=north west][inner sep=0.75pt]  [font=\Large] [align=left] {...};
\draw (268.75,143.45) node [anchor=north west][inner sep=0.75pt]  [font=\Large,rotate=-91.1] [align=left] {...};
\draw (310.87,63.83) node [anchor=north west][inner sep=0.75pt]  [font=\footnotesize] [align=left] {$\displaystyle R_{1}$};
\draw (316.86,148.64) node [anchor=north west][inner sep=0.75pt]  [font=\footnotesize] [align=left] {$\displaystyle R_{n}$};
\draw (241.01,88.5) node [anchor=north west][inner sep=0.75pt]  [font=\footnotesize] [align=left] {$\displaystyle p_{1}$};
\draw (244.38,180.24) node [anchor=north west][inner sep=0.75pt]  [font=\footnotesize] [align=left] {$\displaystyle p_{n}$};
\draw (518.02,164.32) node [anchor=north west][inner sep=0.75pt]  [font=\Large,rotate=-91.1] [align=left] {...};
\draw (491.76,238.23) node [anchor=north west][inner sep=0.75pt]  [font=\footnotesize,color={rgb, 255:red, 208; green, 2; blue, 27 }  ,opacity=1 ] [align=left] {$\displaystyle L_{m}$};
\draw (489.63,198.23) node [anchor=north west][inner sep=0.75pt]  [font=\footnotesize,color={rgb, 255:red, 65; green, 117; blue, 5 }  ,opacity=1 ] [align=left] {$\displaystyle L_{m-1}$};

\end{tikzpicture}

    \caption{The outer rational curves $R_1,...,R_n$ and the outer layers $L_m$ and $L_{m-1}$.}
    \label{fig:layers}
\end{figure}

Now given a function $f \in H^0(V,\O_C)$, we can compute $\res^m_{p_i,R_i}(f\phi|_V) \in \O_S(-\delta)$, where $\delta$ is the value of $\lambda$ on the layer, $L_m$. Define the map $\res^m : H^0(V,\O_C) \to \O_S(-\delta)$ by 
\[
\res^m(f\phi) = \sum_{p_i} \res^m_{p_i,R_i}(f\phi|_V).
\]

\begin{proposition}
    The map $\res^m$ is an $A$-derivation. 
\end{proposition}

\begin{proof}
    The $A$-linearity is clear from definition, so we need only show the Leibniz rule holds. We can reduce to showing this is true when taking the residue at a single $p_i$, since summing over all $p_i$ will still preserve the Leibniz rule. Thus we must show that given $f,g \in H^0(V,\O_C)$, $\res^m_{p_i,R_i}(fg\phi|_V) = f(p_i)\res^m_{p_i,R_i}(g\phi|_V) + g(p_i)\res^m_{p_i,R_i}(f\phi|_V)$. 

    Letting $x$ be a local coordinate for $R_i$ at $p_i$, we can write 
    \begin{align*}
        (f_{p_i})|_{R_i} &= c_0 + c_1x + \cdots \\
        (g_{p_i})|_{R_i} &= d_0 + d_1x + \cdots \\
        (\phi_{p_i})|_{R_i} &= [t](\gamma_{-2}x^{-2} + \gamma_0 + \gamma_1x + \cdots).
    \end{align*}
    Note that there is no $x_i^{-1}$ term as the total residue of $\phi|_{R_i}$ must be 0. A computation then yields
    \begin{align*}
        \res^m_{p_i,R_i}(fg\phi|_V) &= [t](c_0d_1 + d_0c_1) \\
        &= f(p_i)\res^m_{p_i,R_i}(g\phi|_V) + g(p_i)\res^m_{p_i,R_i}(f\phi|_V),
    \end{align*}
    as desired.
\end{proof}

\begin{lemma}\label{splitting}
    The image of the map $\res^m : H^0(V,\O_C) \to \O_S(-\delta)$ is generated by $\mathrm{Ann}_A(t)$ and we have a splitting $[t]\mathrm{Ann}_A(t) \to H^0(V,\O_C)$.
\end{lemma}

\begin{proof}
    First, note that $\O_S(-\delta) \simeq [t]A$. Now, the only condition on the map $H^0(V,\O_C) \to \O_S(-\delta)$ is that the composition $H^0(V,\O_C) \to \O_S(-\delta) \to \O_S$ must be 0, as the total residue (in $\O_S$) of any function must be 0 (from the Mittag-Leffler problem, see, e.g., \cite{acgh1} pgs. 13-15). Let $f \in H^0(V,\O_C)$ and assume we have local expansions at each $p_i$
    \[
    f_{p_i} = c_0^i + c_1^ix_i + \cdots
    \]
    and 
    \[
    \phi_{p_i} = [t](\gamma_{-2}^ix_i^{-2} + \gamma_0^i + \cdots).
    \]
    Putting these together, we have
    \[
    \res^m(f\phi) = [t]\sum_{i=1}^n \gamma_{-2}^ic_1^i \in \O_S(-\delta).
    \]
    After mapping to $\O_S$, we get the following condition
    \[
    t\sum_{i=1}^n \gamma_{-2}^ic_1^i = 0 \in \O_S.
    \]
    Therefore, we see that the only condition on $f$ is that $\sum_{i=1}^n \gamma_{-2}^ic_1^i \in \mathrm{Ann}_A(t)$, so the image of the map lies in $[t]\mathrm{Ann}_A(t)$.

    To show that the image is all of $[t]\mathrm{Ann}_A(t)$, we define a splitting $[t]\mathrm{Ann}_A(t) \to H^0(V,\O_C)$. Let $c \in \mathrm{Ann}_A(t)$. We then define a function $f \in H^0(V,\O_C)$ by $f_{p_1} = (c/\gamma_{-2}^1)x_1$ and $f_{p_i} = 0$ for $i \neq 1$, which one can check is a splitting.
\end{proof}

We then define $f$ to be in $\O_{\widebar{C}}(U)$ if it satisfies
\begin{equation}\label{res_cond}
    \res^m(f\phi) = \sum_{p_i} \res^m_{p_i,R_i}(f\phi|_V) = 0 \in \O_S(-\delta).
\end{equation}
Note that $\O_{\widebar{C}}(U)$ is a ring, since it is the kernel of a derivation.

\begin{remark}
    When $U = \widebar{C}$, we have $\O_{\widebar{C}}(\widebar{C}) = A$. This follows since $\O_{\widebar{C}}(\widebar{C})$ is the subset of $\O_C(C)$ with zero residue, but $\O_C(C)$ is only constants, which all have zero residue. 
\end{remark}

\begin{definition}\label{contraction_def}
    Let $\pi : C \to S$ be a family of centrally aligned genus one curves with $S$ the spectrum of a local artinian ring, $A$. The contraction $\widebar{\pi} : \widebar{C} \to S$ is defined to be the ringed space $\widebar{C} = (C',\O_{\widebar{C}})$ with the map $\tau : C \to \widebar{C}$ topologically being the contraction and on sheaves being the map $\O_{\widebar{C}} \hookrightarrow \tau'_*\O_C = \tau_*\O_C$. 
\end{definition}

\begin{remark}
    We could carry out this construction at any layer $L_i$ by using $\res^i$ instead of $\res^m$. Thus we have a contraction for each layer $L_i$, but we will only be concerned with the contraction at $L_m$.
\end{remark}

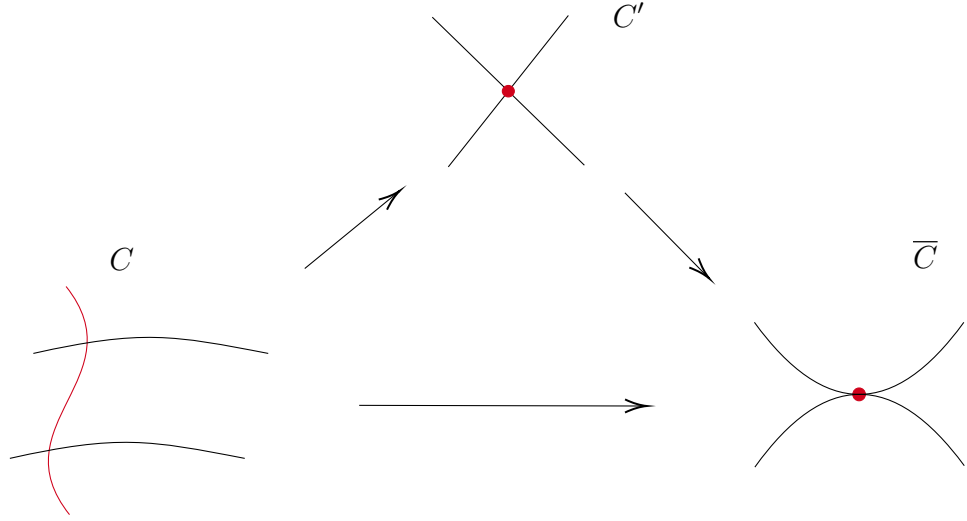
\begin{figure}
    \centering
    \begin{tikzpicture}[x=0.75pt,y=0.75pt,yscale=-1,xscale=1]

\draw [color={rgb, 255:red, 208; green, 2; blue, 27 }  ,draw opacity=1 ]   (112.22,157.02) .. controls (148.25,202.32) and (77.91,226.23) .. (113.94,271.52) ;
\draw    (95.79,190.61) .. controls (150.41,178.39) and (166.12,181.45) .. (213.67,190.61) ;
\draw    (84,243.28) .. controls (138.62,231.07) and (154.33,234.12) .. (201.88,243.28) ;
\draw    (296,21.85) -- (372.33,96.15) ;
\draw    (364.3,21) -- (304.04,97) ;
\draw  [color={rgb, 255:red, 208; green, 2; blue, 27 }  ,draw opacity=1 ][fill={rgb, 255:red, 208; green, 2; blue, 27 }  ,fill opacity=1 ] (331.15,59) .. controls (331.15,57.41) and (332.5,56.12) .. (334.17,56.12) .. controls (335.83,56.12) and (337.18,57.41) .. (337.18,59) .. controls (337.18,60.59) and (335.83,61.88) .. (334.17,61.88) .. controls (332.5,61.88) and (331.15,60.59) .. (331.15,59) -- cycle ;
\draw  [color={rgb, 255:red, 208; green, 2; blue, 27 }  ,draw opacity=1 ][fill={rgb, 255:red, 208; green, 2; blue, 27 }  ,fill opacity=1 ] (507.06,211.16) .. controls (507.06,209.38) and (508.47,207.95) .. (510.22,207.95) .. controls (511.96,207.95) and (513.37,209.38) .. (513.37,211.16) .. controls (513.37,212.93) and (511.96,214.37) .. (510.22,214.37) .. controls (508.47,214.37) and (507.06,212.93) .. (507.06,211.16) -- cycle ;
\draw   (457.67,175) .. controls (492.7,223.21) and (527.73,223.21) .. (562.77,175) ;
\draw   (563,246.96) .. controls (527.65,198.98) and (492.62,199.22) .. (457.9,247.67) ;
\draw    (232,148) -- (278.12,110.27) ;
\draw [shift={(279.67,109)}, rotate = 140.71] [color={rgb, 255:red, 0; green, 0; blue, 0 }  ][line width=0.75]    (10.93,-3.29) .. controls (6.95,-1.4) and (3.31,-0.3) .. (0,0) .. controls (3.31,0.3) and (6.95,1.4) .. (10.93,3.29)   ;
\draw    (392.67,110) -- (433.6,151.57) ;
\draw [shift={(435,153)}, rotate = 225.45] [color={rgb, 255:red, 0; green, 0; blue, 0 }  ][line width=0.75]    (10.93,-3.29) .. controls (6.95,-1.4) and (3.31,-0.3) .. (0,0) .. controls (3.31,0.3) and (6.95,1.4) .. (10.93,3.29)   ;
\draw    (259.33,216.67) -- (401.67,217) ;
\draw [shift={(403.67,217)}, rotate = 180.13] [color={rgb, 255:red, 0; green, 0; blue, 0 }  ][line width=0.75]    (10.93,-3.29) .. controls (6.95,-1.4) and (3.31,-0.3) .. (0,0) .. controls (3.31,0.3) and (6.95,1.4) .. (10.93,3.29)   ;

\draw (132.67,136.33) node [anchor=north west][inner sep=0.75pt]   [align=left] {$\displaystyle C$};
\draw (385.33,13) node [anchor=north west][inner sep=0.75pt]   [align=left] {$\displaystyle C'$};
\draw (536,132.33) node [anchor=north west][inner sep=0.75pt]   [align=left] {$\displaystyle \overline{C}$};

\end{tikzpicture}

    \caption{The construction of $\widebar{C}$.}
    \label{fig:enter-label}
\end{figure}

The map $\widebar{\pi}$ is defined by the same map $A \to A = \O_{\widebar{C}}(\widebar{C})$ that defines $\pi$, and since $\O_{\widebar{C}}(\widebar{C}) \to \O_C(C)$ is the identity, we have the following commuting diagram:
\begin{center}
    \begin{tikzcd}
C \arrow[rr, "\tau"] \arrow[rd, "\pi"'] &   & \widebar{C} \arrow[ld, "\widebar{\pi}"] \\
                                        & S &                                          
\end{tikzcd}
\end{center}

We now collect some facts about $\widebar{C}$ and residues. 

First, let $B = A/\mm^i$ be a quotient of $A$ by a power of its maximal ideal and set $T = \Spec(B)$ and $C_T = C \times_S T$. For any value of $\lambda$, say $u$, we have an isomorphism $\O_T(-u) \simeq \O_S(-u) \otimes_A B$. Indeed, given a map of log schemes $f : T \to S$, we have a natural isomorphism $\O_T(-u) \simeq f^*\O_S(-u) = \O_S(-u) \otimes_A B$.

Next, let $F : C_T \to C$ be the natural map. Since $F^*\omega_C \simeq \omega_{C_T}$ (see \cite{stacks-project} Tag 0E6R) and $F^*\O_C(-\lambda) \simeq \O_{C_T}(-\lambda)$, we see that $F^*\omega_C(-\lambda) \simeq \omega_{C_T}(-\lambda)$. We have a natural map (using the $R^0$ case of base change) 
\[
H^0(C,\omega_C(-\lambda)) \otimes_A B \to H^0(C_T,\omega_{C_T}(-\lambda)).
\]
Then identifying $H^0(C,\omega_C(-\lambda)) \otimes_A B \simeq H^0(C,\omega_C(-\lambda))/\mm^iH^0(C,\omega_C(-\lambda))$, we see that we have a surjective map $H^0(C,\omega_C(-\lambda)) \to H^0(C_T,\omega_{C_T}(-\lambda))$. Now if $\phi$ is a generator of $H^0(C,\omega_C(-\lambda))$ and $\psi$ is a generator of $H^0(C_T,\omega_{C_T}(-\lambda))$, then the image of $\phi$ under the above map will differ from $\psi$ by a multiple of a unit, so their residues will only differ by a multiple of a unit.

\begin{lemma}\label{res_lift}
    Let $P_j \in R_j$ be a node connecting $L_i$ to $L_{i-1}$ in $C_0$ and let $\lambda$ have the value $u$ at layer $L_i$. Let $B = A/\mm^l$ be a quotient of $A$ by a power of its maximal ideal, $T = \Spec(B)$, and $C_T = C \times_S T$. If $g \in \O_C(V)$ is a lift of $f \in \O_{C_T}(V)$ and $\phi \in H^0(C,\omega_C(-\lambda))$ is a lift of $\psi \in H^0(C_T,\omega_{C_T}(-\lambda))$, then $\res^i_{P_j,R_j}(g\phi|_V)$ maps to $\res^i_{P_j,R_j}(f\psi|_V)$ under the map $\O_S(-u) \to \O_T(-u)$. 
\end{lemma}

\begin{proof}
    Let $n$ be the smallest integer such that $\mm^n = 0$. We show the lemma is true for $l = n-1$. This implies the general result, since one can apply the lemma repeatedly to get to any $l$.
    
    This is a computation in local coordinates. Let $x$ be a local coordinate for $R_j$. Let $t_1 \cdots t_i$ be the product of the smoothing parameters up to $L_i$. First, we write
    \[
    (\psi)_{P_j}|_{R_j} = [t_1 \cdots t_i](\gamma_{-2}x^{-2} + \gamma_0 + \cdots)\,dx
    \]
    \[
    (f)_{P_j} = c + c_1x + \cdots
    \]
    from which we compute
    \[
    \res^i_{P_j,R_j}(f\psi|_V) = [t_1 \cdots t_i]\gamma_{-2}c_1.
    \]
    Now since $g$ and $\phi$ are lifts of $f$ and $\psi$, respectively, we can write
    \[
    (\phi)_{P_j}|_{R_j} = [t_1 \cdots t_i]((\gamma_{-2}+v_{-2})x^{-2} + (\gamma_0 + v_0) + \cdots)\,dx
    \]
    \[
    (g)_{P_j} = (c + u_0) + (c_1 + u_1)x + \cdots
    \]
    where $u_i,v_i \in \mm^{n-1}$. From this we compute
    \[
    \res^i_{P_j,R_j}(g\phi|_V) = [t_1 \cdots t_i](\gamma_{-2}c_1 + u_1\gamma_{-2} + v_{-2}c_1 + v_{-2}u_1), 
    \]
    so we see that
    \[
    \res^i_{P_j,R_j}(f\psi|_V) = \res^i_{P_j,R_j}(g\phi|_V) \pmod{\mm^{n-1}}
    \]
    as desired.
\end{proof}

Next, let $S = \Spec(R)$, where $R$ is a local noetherian ring with maximal ideal $\mm$. Let $S_n = \Spec(R/\mm^{n+1})$ and $C_n = C \times_S S_n$. Since $R/\mm^{n+1}$ is a local artinian ring, we have the following diagram of cartesian squares
\begin{center}
    \begin{tikzcd}
C_0 \arrow[r] \arrow[d] & C_1 \arrow[r] \arrow[d] & C_2 \arrow[d] \arrow[r] & \cdots \\
S_0 \arrow[r]           & S_1 \arrow[r]           & S_2 \arrow[r]           & \cdots
\end{tikzcd}
\end{center}
and a contraction for each $i$, $\tau_i : C_i \to \widebar{C_i}$, defined using a differential $\phi_i \in H^0(C_i,\omega_{C_i}(-\lambda))$.

Based on the conclusion of the paragraph before Lemma \ref{res_lift}, we will assume that $\phi_{n+1}$ is a lift of $\phi_n$, since any lift of $\phi_n$ will only differ from $\phi_{n+1}$ by a unit multiple, and we are only concerned with the residue being zero or not. 

\begin{lemma}\label{node_pass}
    Let $q_1,...,q_l$ be the nodes connecting layer $L_0$ to layer $L_1$, $p_1,...,p_n$ the nodes connecting the outermost layer of rational curves $R_1,...,R_n$ to the next layer in, and let $f_n \in \O_{C_n}(V)$. Then 
    \[
    \sum_{p_i} \res_{p_i,R_i}(f_n\phi_n|_V) = -\sum_{q_i} \res_{q_i,L_0}(f_n\phi_n|_V).
    \]
\end{lemma}

\begin{proof}
    Combining Lemmas \ref{invert} and \ref{res_thm}, we see that if $P_1,...,P_k$ are the nodes connecting layer $L_{m-1}$ to $L_{m-2}$, say node $P_j$ is on the curve $Y_j$ in $L_{m-1}$, then 
    \[
    \sum_{p_i} \res_{p_i,R_i}(f_n\phi_n|_V) = -\sum_{p_i} \res_{p_i,Y_i}(f_n\phi_n|_V) =  \sum_{P_j} \res_{P_j,Y_j}(f_n\phi_n|_V).
    \]
    We can then repeat this for the nodes connecting $L_{m-2}$ to $L_{m-3}$, and so forth, until we hit the core $L_0$, where the corresponding equation becomes
    \[
    \sum_{p_i} \res_{p_i,R_i}(f_n\phi_n|_V) = -\sum_{q_i} \res_{q_i,L_0}(f_n\phi_n|_V),
    \]
    as desired.
\end{proof}

Before moving on to the main theorem, we provide an example to illustrate the idea of part of the proof and motivate the statement (and, hopefully, the definition of $\widebar{C}$). 

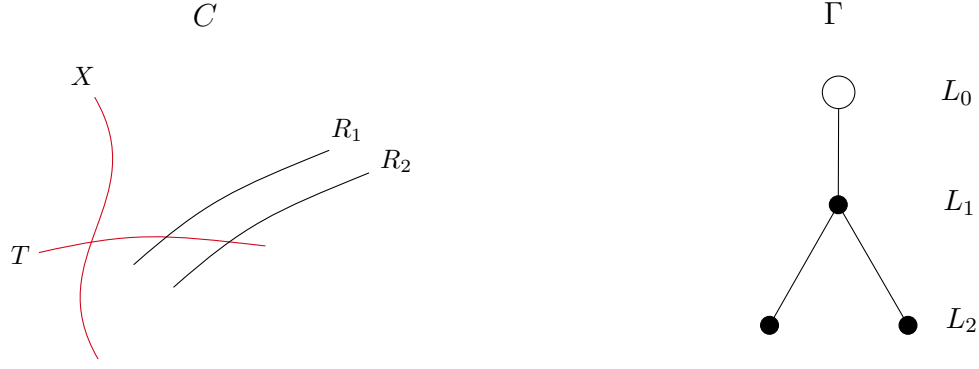
\begin{figure}
    \centering
    \begin{tikzpicture}[x=0.75pt,y=0.75pt,yscale=-1,xscale=1]

\draw [color={rgb, 255:red, 208; green, 2; blue, 27 }  ,draw opacity=1 ]   (94.62,53.01) .. controls (125.17,105.18) and (65.83,132.27) .. (96.38,184.45) ;
\draw [color={rgb, 255:red, 208; green, 2; blue, 27 }  ,draw opacity=1 ]   (66.67,131) .. controls (113.67,120.33) and (128.33,121.67) .. (180.33,127.67) ;
\draw [color={rgb, 255:red, 0; green, 0; blue, 0 }  ,draw opacity=1 ]   (114.2,137.1) .. controls (150.33,105.2) and (163.84,99.34) .. (212.34,79.66) ;
\draw [color={rgb, 255:red, 0; green, 0; blue, 0 }  ,draw opacity=1 ]   (134.2,148.44) .. controls (170.33,116.54) and (183.84,110.67) .. (232.34,91) ;
\draw   (459.83,50.58) .. controls (459.83,46.07) and (463.49,42.42) .. (468,42.42) .. controls (472.51,42.42) and (476.17,46.07) .. (476.17,50.58) .. controls (476.17,55.09) and (472.51,58.75) .. (468,58.75) .. controls (463.49,58.75) and (459.83,55.09) .. (459.83,50.58) -- cycle ;
\draw  [fill={rgb, 255:red, 0; green, 0; blue, 0 }  ,fill opacity=1 ] (463.33,107) .. controls (463.33,104.51) and (465.35,102.5) .. (467.83,102.5) .. controls (470.32,102.5) and (472.33,104.51) .. (472.33,107) .. controls (472.33,109.49) and (470.32,111.5) .. (467.83,111.5) .. controls (465.35,111.5) and (463.33,109.49) .. (463.33,107) -- cycle ;
\draw  [fill={rgb, 255:red, 0; green, 0; blue, 0 }  ,fill opacity=1 ] (498.33,167.5) .. controls (498.33,165.01) and (500.35,163) .. (502.83,163) .. controls (505.32,163) and (507.33,165.01) .. (507.33,167.5) .. controls (507.33,169.99) and (505.32,172) .. (502.83,172) .. controls (500.35,172) and (498.33,169.99) .. (498.33,167.5) -- cycle ;
\draw  [fill={rgb, 255:red, 0; green, 0; blue, 0 }  ,fill opacity=1 ] (428.83,167.5) .. controls (428.83,165.01) and (430.85,163) .. (433.33,163) .. controls (435.82,163) and (437.83,165.01) .. (437.83,167.5) .. controls (437.83,169.99) and (435.82,172) .. (433.33,172) .. controls (430.85,172) and (428.83,169.99) .. (428.83,167.5) -- cycle ;
\draw    (468,58.75) -- (467.83,107) ;
\draw    (467.83,107) -- (433.33,167.5) ;
\draw    (467.83,107) -- (502.83,167.5) ;

\draw (142.67,5) node [anchor=north west][inner sep=0.75pt]   [align=left] {$\displaystyle C$};
\draw (51.33,126.33) node [anchor=north west][inner sep=0.75pt]  [font=\footnotesize] [align=left] {$\displaystyle T$};
\draw (81.33,36.67) node [anchor=north west][inner sep=0.75pt]  [font=\footnotesize] [align=left] {$\displaystyle X$};
\draw (212,62.67) node [anchor=north west][inner sep=0.75pt]  [font=\footnotesize] [align=left] {$\displaystyle R_{1}$};
\draw (236.67,78) node [anchor=north west][inner sep=0.75pt]  [font=\footnotesize] [align=left] {$\displaystyle R_{2}$};
\draw (459.33,4) node [anchor=north west][inner sep=0.75pt]   [align=left] {$\displaystyle \Gamma $};
\draw (518,43.33) node [anchor=north west][inner sep=0.75pt]  [font=\small] [align=left] {$\displaystyle L_{0}$};
\draw (520.5,157.67) node [anchor=north west][inner sep=0.75pt]  [font=\small] [align=left] {$\displaystyle L_{2}$};
\draw (519.33,98.5) node [anchor=north west][inner sep=0.75pt]  [font=\small] [align=left] {$\displaystyle L_{1}$};

\end{tikzpicture}

    \caption{The curve $C$ and its associated level graph $\Gamma$.}
    \label{fig:1}
\end{figure}

\begin{example}\label{mittag-leffler}
    Set $A = k[u]/(u^3)$ and let $C \to \Spec(A)$ be the curve with $C_0$ consisting of a genus one curve, $X$, attached to a single rational curve, $T$, that is attached to two more rational curves, $R_1$ and $R_2$, and assume that every node has smoothing parameter equal to $u$. See Figure \ref{fig:1} for the curve and its associated level graph $\Gamma$. Let $U$ be a neighborhood of the contracted subcurve, and let $q$ be the node connecting $X$ and $T$. We will examine the condition a function $f \in \O_{C_0}(U)$ must satisfy if it lifts to a function $\tilde{f} \in \O_C(U)$. 

    Assume we have such a lift $\tilde{f}$. At the outer nodes $p_1$ and $p_2$, let us write
    \[
    f_{p_i} = c + c_1^ix_i + \cdots.
    \]
    The stalks of $\tilde{f}$ must be lifts of these, so we can write
    \[
    \tilde{f}_{p_i} = (c + du + eu^2) + (c_1^i + d_1^iu + e_1^iu^2)x_i + \cdots,
    \]
    for some $d,e,d_j^i,e_j^i \in k$. Also, for $\phi \in H^0(C_0,\omega_{C_0}(-\lambda))$ and a lift $\tilde{\phi} \in H^0(C,\omega_C(-\lambda))$, we can write
    \begin{align*}
        \phi_{p_i}|_{R_i} &= [u^2](\gamma_{-2}^ix_i^{-2} + \cdots)\,dx_i \\
        \tilde{\phi}_{p_i}|_{R_i} &= [u^2][(\gamma_{-2}^i + \alpha^iu + \beta^iu^2)x_i^{-2} + \cdots]\,dx_i
    \end{align*}
    for some $\alpha^i,\beta^i \in k$. In order for $\tilde{f}$ to be a lift of $f$, the lifts of the stalks must be compatible with the lifts of $f$ on any component of $C_0$ with the nodes removed. Lifting to the rational components is always possible, so the only condition can come from lifting to $X \setminus \{q\}$. 

    Computing and summing the residues of $\tilde{f}\tilde{\phi}$ in $\O(-\delta)$ at $p_1$ and $p_2$, we get 
    \[
    [u^2][(\gamma_{-2}^1 + \alpha^1u + \beta^1u^2)(c_1^1 + d_1^1u + e_1^1u^2) + (\gamma_{-2}^2 + \alpha^2u + \beta^2u^2)(c_1^2 + d_1^2u + e_1^2u^2)].
    \]
    From Lemma \ref{node_pass}, the total residue on $X$ is the negative of the image of this sum under the map to $\O$, and since $X$ is a genus one curve, a compatible lift of $f$ to $\O_C(X \setminus \{q\})$ exists if and only if this residue is 0. This is the Mittag-Leffler problem, which states that, given a set of principal parts of a meromorphic function at points on a curve, there exists a global meromorphic function with these principal parts exactly when the sum of the residues of the principal parts paired with every global differential on the curve is zero. For a detailed explanation, see, e.g., \cite{acgh1} pgs. 13-15. Expanding the products in the sum, we get
    \[
    [u^2][(\gamma_{-2}^1c_1^1 + \gamma_{-2}^2c_1^2) + au + bu^2],
    \]
    where $a$ and $b$ are sums containing the other coefficients. Now since $u^3 = 0$ in $A$, the image of this in $\O$ will just be
    \[
    [u^2](\gamma_{-2}^1c_1^1 + \gamma_{-2}^2c_1^2).
    \]
    Thus for $f$ to lift, it must satisfy
    \[
    u^2(\gamma_{-2}^1c_1^1 + \gamma_{-2}^2c_1^2) = 0,
    \]
    which is equivalent to 
    \[
    \gamma_{-2}^1c_1^1 + \gamma_{-2}^2c_1^2 = 0,
    \]
    since if $\gamma_{-2}^1c_1^1 + \gamma_{-2}^2c_1^2 \neq 0$, then it will not be in the annihilator of $u^2$. Now we see that this is exactly the condition for $f$ to descend to $\widebar{C_0}$. 
\end{example}

\begin{theorem}\label{main}
    Let $R$ be a local noetherian ring complete with respect to its maximal ideal $\mm$, $S = \Spec(R)$, and $S_n = \Spec(R/\mm^{n+1})$. Let $C \to S$ be a family of centrally aligned genus one curves and set $C_n = C \times_S S_n$. Assume $\O_C(-\lambda) \to \O_C$ is injective (e.g., $C$ is a smoothing of $C_0$).
    
    Let $U \subset \widebar{C_n}$ be an open subset, $V = \tau_n^{-1}(U)$, and $f_n \in \O_{C_n}(V)$. The following are equivalent:
    \begin{enumerate}
        \item $f_n$ is in the image of $\O_{C_m}(V) \to \O_{C_n}(V)$ for all $m \geq n$;
        \item there exist compatible lifts $f_m \in \O_{C_m}(V)$ of $f_n$ for all $m \geq n$;
        \item $f_n$ is in $\O_{\widebar{C_n}}(U)$;
        \item $f_n$ satisfies the residue condition: $\sum_{p_i} \res^m_{p_i,R_i}(f_n\phi_n|_V) = 0 \in \O_{S_n}(-\delta)$.
    \end{enumerate}
\end{theorem}

\begin{proof}
    We clearly have (ii) implies (i), and (iii) and (iv) are equivalent by definition of $\widebar{C_n}$. We will first show (i) implies (iv), then finish by showing (iv) implies (ii).

    We begin with (i) implies (iv). If $\tau_n(E) \not\in U$, then $\O_{\widebar{C_n}}(U) = \O_{C_n}(V)$, so (i) clearly implies (iv) in this case. 
    
    Now assume $\tau_n(E) \in U$. Let
    \[
    \sum_{p_i}\res^m_{p_i,R_i}(f_n\phi_n|_V) = [t]a
    \]
    for some $a \in R/\mm^{n+1}$.
    Let $N \geq n$ be the smallest integer such that $ta \in \mm^N$, but $ta \not\in \mm^{N+1}$. Note that we could have $N=n$. From Lemma \ref{res_lift}, we know that 
    \[
    \sum_{p_i} \res^m_{p_i,R_i}(f_N\phi_N|_V) = [t]\Tilde{a} + [t]G,
    \]
    where $G \in \mm^N$ and $\Tilde{a}$ is a lift of $a$ to $R/\mm^{N+1}$. Then we have
    \[
    \sum_{p_i} \res_{p_i,R_i}(f_N\phi_N|_V) = t\Tilde{a} \in R/\mm^{N+1},
    \]
    since $tG \in \mm^{N+1}$. Now applying Lemma \ref{node_pass} with $q_i$ as in said lemma, we have
    \[
    \sum_{q_i} \res_{q_i,L_0}(f_N\phi_N|_V) = -t\Tilde{a},
    \]
    and since this is the total residue of $f_N\phi_N$ on $L_0$, we must have $t\Tilde{a} = 0$, from which we deduce that $\Tilde{a} = 0$, since if $\tilde{a} \neq 0$, then $\Tilde{a} \not\in \text{Ann}(t)$ (this follows from the injectivity of $\O_C(-\lambda) \to \O_C$). Restricting back to $S_n$, we conclude that $a = 0 \in R/\mm^{n+1}$, so $ta = 0$, as desired.
    
    We now address (iv) implies (ii). Let $J = \ker(R/\mm^{n+2} \to R/\mm^{n+1})$. Since $C_{n+1}$ is an infinitesimal deformation of $C_n$, we have a short exact sequence
    \[
    0 \to \O_{C_{n+1}} \otimes J \to \O_{C_{n+1}} \to \O_{C_n} \to 0.
    \]
    By flatness of the maps $C_k \to S_k$, we see that $\O_{C_{n+1}} \otimes J = \O_{C_0} \otimes J$. Now this sequence induces a long exact sequence in cohomology, and we see that the obstruction to lifting $f_n$ to a function $f_{n+1}$ is found in $H^1(C_0,\O_{C_0} \otimes J)$. When $\tau_n(E) \not\in U$, $V$ is a collection of rational curves, so $H^1(V,\O_{C_0} \otimes J|_V) = 0$ and we always have a lift. We can then repeat this on $C_{n+1}$ to find a compatible lift to $C_{n+2}$, then repeat to get a lift to any $C_m$ for $m \geq n$. 

    Now we examine the case when $\tau_n(E) \in U$. We will approach this by showing that if $f_n$ satisfies (iv), then we can find a lift $f_{n+1}$ that also satisfies (iv). We can then inductively lift $f_{n+1}$ to $f_{n+2}$ also satisfying (iv), and so forth to any order. 
    
    We first make a reduction to the problem. If we let $X$ be a component of $C_{n+1}$ with nodes $r_1,...,r_k$, then we claim that finding a lift of $f_n$ to $X$ is equivalent to finding compatible lifts of $f_n|_{X\setminus\{r_1,...,r_k\}}$ and $(f_n)_{r_i} \in (\O_{C_n})_{r_i}$. This follows from Zariski gluing, since viewing $(f_n)_{r_i}$ as a function in a suitably small neighborhood $W_i$ of $r_i$, we see that $X\setminus\{r_1,...,r_k\}$ and $W_1,...,W_k$ form a cover of $X$. Now we can always find lifts to $X \setminus \{r_1,...,r_k\}$, since $X \setminus \{r_1,...,r_k\}$ is affine, and $(\O_{C_{n+1}})_{r_i} \to (\O_{C_{n}})_{r_i}$ is surjective, so we can also find a lift to $(\O_{C_{n+1}})_{r_i}$. However, these lifts must be compatible, and the lift to $(\O_{C_{n+1}})_{r_i}$ can potentially force the lift to $X \setminus \{r_1,...,r_k\}$ to have a zero or pole at $r_i$. Thus we see that finding a lift is equivalent to finding a rational function on $X$ with prescribed zeroes or poles at the nodes. As we saw above, this can always be done when the component is rational. Furthermore, since $C_{n+1}$ has genus one, the dual graph of $C_{n+1}$ minus the vertices corresponding to the core forms a tree. Thus if we think of finding a lift as first lifting to layer $L_m$, then layer $L_{m-1}$, and so forth, we see that the only obstruction to lifting arises from lifting to $L_0$. This is the obstruction in $H^1$ previously discussed.

    Let us now compute this obstruction and show that we can always find a lift when $f_n \in \O_{\widebar{C_n}}(U)$. In a similar fashion to the proof of Lemma \ref{res_lift}, we compute in local coordinates. At each $p_i$, with $x_i$ the local coordinate for $R_i$, we write
     \[
    (\phi_n)_{p_i}|_{R_i} = [t](\gamma_{-2}^{i}x_{i}^{-2} + \gamma_0^i + \cdots)\,dx_i
    \]
    \[
    (f_n)_{p_i} = c^i + c_1^ix_i + \cdots
    \]
    from which we compute
    \[
    \res^m_{p_i,R_i}(f_n\phi_n|_V) = [t]\gamma_{-2}^ic_1^i.
    \]
    Note that $\gamma^i_{-2}$ is a unit, since its restriction to the special fiber $C_0$ is a unit, so $\gamma^i_{-2}$ is a unit plus nilpotents, which is a unit. Again as in Lemma \ref{res_lift}, since $f_{n+1}$ and $\phi_{n+1}$ are lifts of $f_n$ and $\phi_n$, respectively, we can write
    \[
    (\phi_{n+1})_{p_i}|_{R_i} = [t][(\gamma_{-2}^i+v_{-2}^i)x_i^{-2} + (\gamma_0^i + v_0^i) + \cdots]\,dx_i
    \]
    \[
    (f_{n+1})_{p_i} = (c^i + u_0^i) + (c_1^i + u_1^i)x_i + \cdots
    \]
    where $u_j^i,v_j^i \in \mm^{n+1}$. From this we compute
    \begin{align*}
        \res^m_{p_i,R_i}(f_{n+1}\phi_{n+1}|_V) &= [t](\gamma_{-2}^ic_1^i + u_1^i\gamma_{-2}^i + v_{-2}^ic_1^i + u_1^iv_{-2}^i) \\
        &= [t]\gamma_{-2}^ic_1^i + t(u_1^i\gamma_{-2}^i + v_{-2}^ic_1^i + u_1^iv_{-2}^i).
    \end{align*}
    Now using Lemma \ref{node_pass}, we can compute
    \[
    \sum_{q_i} \res_{q_i,L_0}(f_{n+1}\phi_{n+1}|_V) = -t\sum_{p_i} \gamma_{-2}^ic_1^i,
    \]
    where the other terms in the residue vanish since $tu_1^i,tv_{-2}^i \in \mm^{n+2}$. We then have principal parts on $L_0$ and want to know if these principal parts come from a rational function -- this is the Mittag-Leffler problem (see Example \ref{mittag-leffler}). Thus to find our rational function on $L_0$ we only need $t\sum_{p_i} \gamma_{-2}^ic_1^i = 0 \in \O_{S_{n+1}}$, which is satisfied by assumption. Therefore, we can find a lift $f_{n+1}$. Furthermore, we can choose $u_1^i$ to satisfy $u_1^i\gamma_{-2}^i + v_{-2}^ic_1^i + u_1^iv_{-2}^i = 0$ for each $i$ by choosing $u^i = -v_{-2}^ic_1^i(\gamma_{-2}^i + v_{-2}^i)^{-1}$ (note that $\gamma_{-2}^i + v_{-2}^i$ is a unit plus a nilpotent, so is a unit), so we can choose the lift $f_{n+1}$ to satisfy (iv). Hence we can find compatible lifts of $f_n$ to $f_m$ for any $m \geq n$.
\end{proof}

\subsection{Properties of the Contraction}

\begin{proposition}
    Let $B = A/\mm^i$ be a quotient of a local artinian ring $A$ by a power of its maximal ideal, $T = \Spec(B)$, and $C_{T} = C \times_S T$. Then $\widebar{C}_{T} \simeq \widebar{C} \times_S T$, i.e., the following diagram is cartesian
    \begin{center}
        \begin{tikzcd}
\widebar{C}_{T} \arrow[d] \arrow[r] & \widebar{C} \arrow[d] \\
T \arrow[r]                         & S                    
\end{tikzcd}
    \end{center}
\end{proposition}

\begin{proof}
    First, we reduce to the case $B = A/I$, where $I \simeq k$. To do this, we first note that if $B = A/\mm^i$ and if $\mm^{n+1} = 0$ and $\mm^n \neq 0$, then we can write $B$ as a successive quotient of $A$, first quotienting by $\mm^n$, then $\mm^{n-1}$, and so on until reaching $\mm^i$. Thus, we can work with $A/\mm^n$. Now writing $\mm^n = (a_1,...,a_j)$, we can write $A/\mm^n$ as a successive quotient by $(a_1)$, then $(a_2)$, and so on until $(a_j)$. Since $\mm^{n+1} = 0$, we see that each $(a_i) \simeq k$, so we have the desired reduction to $B = A/I$ with $I \simeq k$. Assume $I = (a)$ for some $a \in A$.
    
    The underlying topological spaces of $\widebar{C}_{T}$ and $\widebar{C} \times_S T$ are the same, equal to $C'$, so we need only show that $\O_{\widebar{C}_{T}} = \O_{\widebar{C}} \otimes_{\O_S} \O_T$ as sheaves on $C'$. Away from the contracted point $P = \tau(E)$, this clearly holds, so we need only show that if $U \subset C'$ is an open set containing $P$, then $\O_{\widebar{C}_{T}}(U) = (\O_{\widebar{C}} \otimes_{\O_S} \O_T)(U)$. For this, it suffices to show $\O_{\widebar{C}_{T}}(U) = \O_{\widebar{C}}(U) \otimes_{A} B$. Set $V = (\tau')^{-1}(U)$.

    Using Lemma \ref{splitting}, we have the following commutative diagram with split exact rows:

    \begin{center}
        \begin{tikzcd}
0 \arrow[r] & \O_{\bar{C}}(U) \arrow[r] \arrow[d]           & \O_C(V) \arrow[r] \arrow[d]                              & {[t]}\mathrm{Ann}_A(t) \arrow[r] \arrow[d]                      & 0 \\
0 \arrow[r] & \O_{\bar{C}}(U) \otimes B \arrow[r] \arrow[d] & \O_C(V) \otimes B \arrow[r] \arrow[d] \arrow[ld, dashed] & {[t]}\mathrm{Ann}_A(t) \otimes B \arrow[r] \arrow[d, "0"] & 0 \\
0 \arrow[r] & \O_{\bar{C}_T}(U) \arrow[r]                 & \O_{C_T}(V) \arrow[r]                                    & {[t]}\mathrm{Ann}_B(t) \arrow[r]                                & 0
\end{tikzcd}
    \end{center}
    Our goal is to show that the bottom left vertical map $\O_{\bar{C}}(U) \otimes B \to \O_{\bar{C}_T}(U)$ is an isomorphism. Note that the bottom right vertical map is zero, as $t \in A$ is in some power of $\mm$, so the annihilator in $A$ will become zero upon quotienting by the power of $\mm$ to obtain $B$. 
    
    First, we note that we have a map $\O_C(V) \otimes B \to \O_{\bar{C}_T}(U)$ (dashed in the diagram). To see this, note that $\O_C(V) \to \O_C(V) \otimes B$ is surjective. Let $f \in \O_C(V) \otimes B$ and, by abuse of notation, we also denote a preimage in $\O_C(V)$ by $f$. If $f_T$ is the image of $f$ in $\O_{C_T}(V)$, then $f_T$ maps to zero in $[t]\mathrm{Ann}_B(t)$, as the composition from $\O_C(V) \otimes B \to [t]\mathrm{Ann}_B(t)$ is zero. Thus $f_T$ is actually in $\O_{\bar{C}_T}(U)$, i.e., the map $\O_C(V) \otimes B \to \O_{C_T}(V)$ factors through $\O_{\bar{C}_T}(U)$. By Theorem \ref{main}, any element of $\O_{\bar{C}_T}(U)$ has a lift to $\O_C(V)$, so this map is surjective. Since the maps between rows are maps of split exact sequences, we can see that the map $\O_{\bar{C}}(U) \otimes B \to \O_{\bar{C}_T}(U)$ must also be surjective.

    We now show that $\O_{\bar{C}}(U) \otimes B \to \O_{\bar{C}_T}(U)$ is injective to complete the proof. If we let 
    \begin{align*}
        K_1 &= \ker(\O_{\bar{C}}(U) \otimes B \to \O_{\bar{C}_T}(U)) \\
        K_2 &= \ker(\O_C(V) \otimes B \to \O_{C_T}(V))
    \end{align*}
    then from the Snake Lemma we have the following exact sequence
    \[
    0 \to K_1 \to K_2 \to [t]\mathrm{Ann}_A(t) \otimes B \to 0
    \]
    Therefore, if we can show that $K_2 \to [t]\mathrm{Ann}_A(t) \otimes B$ is injective, then $K_1 = 0$ and we are done. 
    
    Let $f \in K_2$ such that $f$ maps to zero in $[t]\mathrm{Ann}_A(t) \otimes B$. Denoting a preimage of $f$ in $\O_C(V)$ by $\tilde{f}$, we see that $\tilde{f} \in \ker(\O_C(V) \to \O_{C_T}(V))$ and $\tilde{f}$ maps to zero in $[t]\mathrm{Ann}_A(t)$. Thus $\tilde{f} \in (I\O_C)(V)$. We will show that we actually have $\tilde{f} \in I\O_C(V)$, so $f = 0 \in \O_C(V) \otimes B$. Since $I = (a) \simeq k$, we have $I\O_C \simeq \O_{C_0}$, so $\tilde{f}$ corresponds to a function $g \in \O_{C_0}(V)$. This function $g$ will have residue zero, since $\tilde{f}$ did, so by Theorem \ref{main}, $g$ lifts to a function $\tilde{g} \in \O_C(V)$. Now applying the multiplication by $a$ map to $\O_C(V)$, we have $\tilde{f} = a\tilde{g}$. Thus $\tilde{f} \in I\O_C(V)$, and we are done.
\end{proof}

\begin{lemma}\label{finite}
    The map $C' \to \widebar{C}$ is finite.
\end{lemma}

\begin{proof}
    We need only show that if $U \subset C'$ is an open subset containing $P = \tau(E)$, then $\O_{C'}(U)$ is a finitely generated module over $\O_{\widebar{C}}(U)$. 

    Let $M$ be the image of $\res^m : \O_{C'}(U) \to \O_S(-\delta)$. Then we have an exact sequence of $\O_{\widebar{C}}(U)$-modules
    \[
    0 \to \O_{\widebar{C}}(U) \to \O_{C'}(U) \to M \to 0.
    \]
    Now since $\O_S(-\delta)$ is a noetherian $A$-module, $M$ is finitely generated as an $A$-module. Thus, since $A \subset \O_{\widebar{C}}(U)$, $M$ is also finitely generated as an $\O_{\widebar{C}}(U)$-module. Therefore, both $\O_{\widebar{C}}(U)$ and $M$ are finitely generated as $\O_{\widebar{C}}(U)$-modules, so we see that the same is true for $\O_{C'}(U)$. 
\end{proof}

\begin{proposition}
    The ringed space $\widebar{C}$ is a scheme.
\end{proposition}

\begin{proof}
    First, since $\tau : C \to \widebar{C}$ is an isomorphism outside of $E$, we need only show that $\tau(E) = P$ has an affine open neighborhood. We reduce to showing this when $S = \Spec(k)$, since if $U$ is an affine neighborhood of $P$ over $\Spec(k)$, then a deformation of this neighborhood over a local artinian ring will still be affine (\cite{hartdef} Cor. 4.8).

    Let $U$ be an affine open neighborhood of $P$ in $C'$, e.g., by removing a smooth point from each component. We claim that $(U,\O_{\widebar{C}}|_U)$ is an affine scheme, and thus an affine neighborhood of $P$ in $\widebar{C}$. For this to be true, we first show that $U$ is homeomorphic to $\Spec\,\O_{\widebar{C}}(U)$. Since $C'$ and $\widebar{C}$ have the same topological space, we see that this is equivalent to showing that $\Spec\,\O_{C'}(U) \to \Spec\,\O_{\widebar{C}}(U)$ is a homeomorphism. To show this, let $V = U \setminus P$, so $(V, \O_{C'}|_V)$ and $(V, \O_{\widebar{C}}|_V)$ are both affine, equal to $\Spec\,\O_{C'}(V)$ and $\Spec\,\O_{\widebar{C}}(V)$, respectively. We then have the following commutative diagram
    \begin{center}
        \begin{tikzcd}
{\Spec\,\O_{\widebar{C}}(U)} & {\Spec\,\O_{\widebar{C}}(V)} \arrow[l] \\
{\Spec\,\O_{C'}(U)} \arrow[u] & {\Spec\,\O_{C'}(V)} \arrow[u] \arrow[l]
\end{tikzcd}
    \end{center}
    The left vertical map is finite by Lemma \ref{finite}, so in particular it is closed. The right vertical map is a bijection, and since the horizontal maps are just adding one point to each space, the left vertical map must also be a bijection. Thus we have a closed bijection, which is a homeomorphism.

    Lastly, to check that $(U,\O_{\bar{C}}|_U)$ is isomorphic to $\Spec\,\O_{\bar{C}}(U)$ we must check that the maps on stalks are isomorphisms. The only place this can fail is at the point $P$. To check that this is an isomorphism, we consider the exact sequence of sheaves
    \[
    0 \to \O_{\bar{C}} \to \O_{C'} \to M \to 0.
    \]
    Restricting to $U$, we still have an exact sequence, so we see that $(\O_{\bar{C}}|_U)_P = \ker((\O_{C'}|_U)_P \to (M|_U)_P)$. On the other hand, $\O_{\bar{C}}(U)$ is defined to be the kernel of $\O_{C'}(U) \to M(U)$, so we see that the stalks at $P$ are isomorphic. 
\end{proof}

\begin{proposition}\label{flat}
    The map $\widebar{\pi} : \widebar{C} \to S$ is flat.
\end{proposition}

\begin{proof}
    First, since $\pi : C \to S$ is flat and $\tau : C \to \widebar{C}$ is an isomorphism outside of the contracted subcurve, we see that we need only check flatness at the contracted point, i.e., if $P = \tau(E)$, we must show that $A \to \O_{\widebar{C},P}$ is flat. Using the local criterion for flatness over a local artinian ring, see \cite{stacks-project} Tag 051K, we must show that 
    \begin{enumerate}
        \item $\O_{\widebar{C},P}/\mm\O_{\widebar{C},P}$ is flat over $A/\mm$;
        \item $\text{Tor}_1^A(k, \O_{\widebar{C},P}) = 0$.
    \end{enumerate}
    The first condition is trivially satisfied, since $A/\mm \simeq k$ is a field. We now proceed with the proof of (ii). 

    Let us first introduce some notation. Since $A$ is a local artinian ring, its maximal ideal $\mm$ is nilpotent, say $\mm^{n+1} = 0$ for some $n \geq 0$. We write $A_i = A/\mm^{i+1}$ and $C_i = C \times_S \Spec(A_i)$ for $0 \leq i \leq n$, so $A_n = A$ and $C_n = C$. We proceed by induction on $i$. For $i = 0$, $A_0 = k$ is a field, so $\Tor_1^{A_0}(k,\O_{\widebar{C_0},P}) = 0$. Now assume that $\Tor_1^{A_i}(k,\O_{\widebar{C_i},P}) = 0$. We will show that $\Tor_1^{A_{i+1}}(k,\O_{\widebar{C_{i+1}},P}) = 0$. 

    We will make use of the spectral sequence (\cite{stacks-project} Tag 068F)
    \[
    \Tor_n^{A_i}(k,\Tor_m^{A_{i+1}}(A_i,\O_{\widebar{C_{i+1}},P})) \implies \Tor_{n+m}^{A_{i+1}}(k,\O_{\widebar{C_{i+1}},P}).
    \]
    We will show $\Tor_1^{A_i}(k,\Tor_0^{A_{i+1}}(A_i,\O_{\widebar{C_{i+1}},P})) = 0$ and $\Tor_0^{A_i}(k,\Tor_1^{A_{i+1}}(A_i,\O_{\widebar{C_{i+1}},P})) = 0$, which together imply that $\Tor_1^{A_{i+1}}(k,\O_{\widebar{C_{i+1}},P}) = 0$, as desired. First, using the induction hypothesis and the definition of $\Tor_0$, we have
    \[
    \Tor_1^{A_i}(k,\Tor_0^{A_{i+1}}(A_i,\O_{\widebar{C_{i+1}},P})) = \Tor_1^{A_i}(k,\O_{\widebar{C_i},P}) = 0.
    \]
    Next, let us denote $J = \ker(A_{i+1} \to A_i)$, so we have an exact sequence 
    \[
    0 \to J \otimes \O_{C_{i+1}} \to \O_{C_{i+1}} \to \O_{C_i} \to 0,
    \]
    since $C_{i+1}$ is an infinitesimal deformation of $C_i$. Pushing forward along the topological contraction $\tau'$, we get
    \[
    0 \to J \otimes \tau'_*\O_{C_{i+1}} \to \tau'_*\O_{C_{i+1}} \to \tau'_*\O_{C_i}.
    \]
    From Theorem \ref{main}, we see that this sequence gives rise to
    \[
    0 \to J \otimes \O_{\widebar{C_{i+1}}} \to \O_{\widebar{C_{i+1}}} \to \O_{\widebar{C_{i}}} \to 0.
    \]
    Taking stalks at $P$, we have the exact sequence of $A_{i+1}$-modules
    \begin{align}\label{O-seq}
    0 \to J \otimes \O_{\widebar{C_{i+1}},P} \to \O_{\widebar{C_{i+1}},P} \to \O_{\widebar{C_{i}},P} \to 0.
    \end{align}
    We note that this is the exact sequence that arises from tensoring
    \begin{align}\label{A-seq}
        0 \to J \to A_{i+1} \to A_i \to 0
    \end{align}
    with $\O_{\widebar{C_{i+1}},P}$. Now since \eqref{O-seq} and \eqref{A-seq} are exact, the long exact sequence from $\Tor_m^{A_{i+1}}(-,\O_{\widebar{C_{i+1}},P})$ applied to \eqref{A-seq} tells us that $\Tor_1^{A_{i+1}}(A_i,\O_{\widebar{C_{i+1}},P}) = 0$. This then implies that
    \[
    \Tor_0^{A_i}(k,\Tor_1^{A_{i+1}}(A_i,\O_{\widebar{C_{i+1}},P})) = \Tor_0^{A_i}(k,0) = 0,
    \]
    which concludes the proof.
\end{proof}

\begin{proposition}
    The map $\widebar{\pi} : \widebar{C} \to S$ is proper.
\end{proposition}

\begin{proof}
    Since $\pi$ is proper and $\tau$ is surjective, it suffices to show that $\widebar{\pi}$ is separated and locally of finite type (see \cite{stacks-project} Tag 03GN). First, since $C' \to \widebar{C}$ is finite and $C' \to S$ is locally of finite type, $\widebar{\pi} : \widebar{C} \to S$ is locally of finite type (see \cite{atiyahmacdonald} Prop. 7.8). Next, for separatedness we reduce to the case where $S = \Spec(k)$, since we can check separatedness on the reduction of a scheme. 

    Now we note that the map $\tau : C \to \widebar{C}$ can be factored as $C \to C' \to \widebar{C}$. The map $C \to C'$ is proper, as both $C$ and $C'$ are proper over $S$ (\cite{stacks-project} Tag 01W6), and $C' \to \widebar{C}$ is finite, so also proper. Therefore, $\tau : C \to \widebar{C}$ is proper. In particular, $\tau$ is surjective and universally closed, so $\widebar{\pi} : \widebar{C} \to S$ is separated by \cite{stacks-project} Tag 09MQ (statement (4)).
\end{proof}

\begin{proposition}\label{sing}
    The curve $\widebar{C_0}$ has a Gorenstein genus one singularity with $n$ branches at the point $\tau(E)$. 
\end{proposition}

\begin{proof}
    Let $P = \tau'(E) \in C'$. Since $C'$ consists of $n$ rational curves meeting transversally at the point $P$, we see that $\hat\O_{C',P} \simeq k[[x_1]] \times_k \cdots \times_k k[[x_n]]$ for some parameters $x_1,...,x_n$. 
    
    Let $f$ be a function on $C_0$ that descends to $\widebar{C_0}$, which we can view as an element of $\hat\O_{\widebar{C_0},P}$. This function must necessarily be constant on the interior of the aligned circle, since $E$ is proper and reduced. Thus, choosing local coordinates $x_i$, $y_i$ at each $p_i$ such that $x_i$ is the coordinate for $R_i$, we may write $f_{p_i} = c + c_1^ix_i + \cdots \in \hat\O_{C_0,p_i}$. 

    Now $\phi|_{R_i} \in H^0(R_i, \omega_{R_i}(2p_i) \otimes \pi^*\O_k(-\delta))$, so we can express $\phi_{p_i}|_{R_i}$ in these coordinates as 
    \[
    \phi_{p_i}|_{R_i} = [t](\gamma_{-2}^ix_i^{-2} + \gamma_0^i + \gamma_1^ix_i + \cdots)\,dx_i.
    \] 
    Note that there is no $x_i^{-1}$ term as the total residue of $\phi|_{R_i}$ must be 0, and $\gamma_{-2}^i \neq 0$, since $\phi|_{R_i}$ has a pole of order two at $p_i$. Therefore, the residue condition \eqref{res_cond} tells us that 
    \begin{align}\label{zero-sum}
         \sum_{p_i} \gamma_{-2}^ic_1^i = 0.
    \end{align}

   Now we send such an $f$ to $\hat\O_{C',P}$ via the inclusion $\hat\O_{\widebar{C_0},P} \to \hat\O_{C',P}$ by sending $f$ to $(f_{p_1},...,f_{p_n}) \in k[[x_1]] \times_k \cdots \times_k k[[x_n]]$. As we have just seen, $f$ descending to a function on $\widebar{C_0}$ is equivalent to \eqref{zero-sum} being satisfied. This condition cuts out the subring of $k[[x_1]] \times_k \cdots \times_k k[[x_n]]$ defining a Gorenstein genus one singularity with $n$ branches (see the appendix of \cite{smyth_mstable}), which proves the statement.
\end{proof}

\begin{example}\label{tacnode}
    Set $A = k[u]/(u^3)$ and let $C \to \Spec(A)$ be the curve with $C_0$ consisting of a genus one curve, $X$, attached to a single rational curve, $T$, that is attached to two more rational curves, $R_1$ and $R_2$, and assume that every node has smoothing parameter equal to $u$. See Figure \ref{fig:1} for the curve and its associated level graph $\Gamma$. We wish to contract $E = X \cup T$, which is colored red in Figure \ref{fig:1} (the contraction is also shown in Figure \ref{fig:2}). From Proposition \ref{sing}, the curve $\widebar{C_0}$ should have a tacnode at $\tau(E)$. Let us compute this explicitly. 

    Let $p_1$ and $p_2$ be the nodes connecting $R_1$ and $R_2$ to $T$, respectively, and assume we have local coordinates $x_1$ for $R_1$ at $p_1$ and $x_2$ for $R_2$ at $p_2$. As in the proof of Proposition \ref{sing}, given any function $f$ on an open neighborhood of $E$ in $C_0$, we can write 
    \[
    f_{p_i} = c + c_1^ix_i + \cdots
    \]
    and for the differential $\phi \in H^0(C_0,\omega_{C_0}(-\lambda))$
    \[
    \phi_{p_i}|_{R_i} = [u^2](\gamma_{-2}^ix_i^{-2} + \gamma_0^i + \gamma_1^ix_i + \cdots)\,dx_i.
    \]
    Note that $t = u^2$ is the product of all smoothing parameters for this example. The condition for $f$ to descend to the contraction is then
    \[
    [u^2](\gamma_{-2}^1c_1^1 + \gamma_{-2}^2c_1^2) = 0 \in \O_k(-\delta).
    \]
    This is equivalent to 
    \[
    \gamma_{-2}^1c_1^1 + \gamma_{-2}^2c_1^2 = 0,
    \]
    which is the equation of a tacnode. 
\end{example}

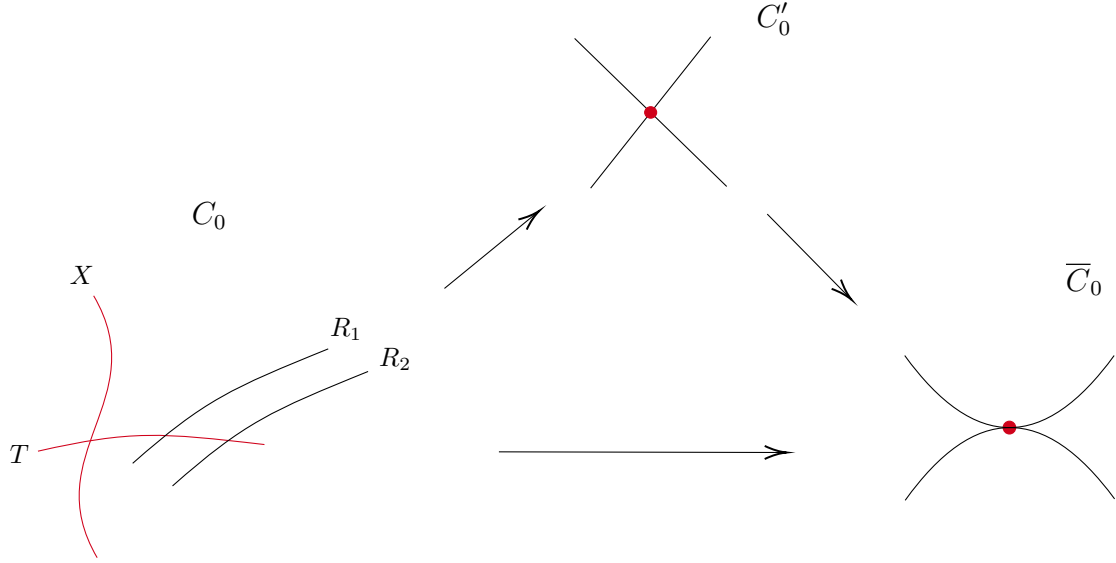
\begin{figure}
    \centering
    \begin{tikzpicture}[x=0.75pt,y=0.75pt,yscale=-1,xscale=1]

\draw [color={rgb, 255:red, 208; green, 2; blue, 27 }  ,draw opacity=1 ]   (68.62,167.01) .. controls (99.17,219.18) and (39.83,246.27) .. (70.38,298.45) ;
\draw [color={rgb, 255:red, 208; green, 2; blue, 27 }  ,draw opacity=1 ]   (40.67,245) .. controls (87.67,234.33) and (102.33,235.67) .. (154.33,241.67) ;
\draw [color={rgb, 255:red, 0; green, 0; blue, 0 }  ,draw opacity=1 ]   (88.2,251.1) .. controls (124.33,219.2) and (137.84,213.34) .. (186.34,193.66) ;
\draw [color={rgb, 255:red, 0; green, 0; blue, 0 }  ,draw opacity=1 ]   (108.2,262.44) .. controls (144.33,230.54) and (157.84,224.67) .. (206.34,205) ;
\draw    (310,37.85) -- (386.33,112.15) ;
\draw    (378.3,37) -- (318.04,113) ;
\draw  [color={rgb, 255:red, 208; green, 2; blue, 27 }  ,draw opacity=1 ][fill={rgb, 255:red, 208; green, 2; blue, 27 }  ,fill opacity=1 ] (345.15,75) .. controls (345.15,73.41) and (346.5,72.12) .. (348.17,72.12) .. controls (349.83,72.12) and (351.18,73.41) .. (351.18,75) .. controls (351.18,76.59) and (349.83,77.88) .. (348.17,77.88) .. controls (346.5,77.88) and (345.15,76.59) .. (345.15,75) -- cycle ;
\draw  [color={rgb, 255:red, 208; green, 2; blue, 27 }  ,draw opacity=1 ][fill={rgb, 255:red, 208; green, 2; blue, 27 }  ,fill opacity=1 ] (525.06,233.16) .. controls (525.06,231.38) and (526.47,229.95) .. (528.22,229.95) .. controls (529.96,229.95) and (531.37,231.38) .. (531.37,233.16) .. controls (531.37,234.93) and (529.96,236.37) .. (528.22,236.37) .. controls (526.47,236.37) and (525.06,234.93) .. (525.06,233.16) -- cycle ;
\draw   (475.67,197) .. controls (510.7,245.21) and (545.73,245.21) .. (580.77,197) ;
\draw   (581,268.96) .. controls (545.65,220.98) and (510.62,221.22) .. (475.9,269.67) ;
\draw    (244.67,163.33) -- (290.79,125.6) ;
\draw [shift={(292.33,124.33)}, rotate = 140.71] [color={rgb, 255:red, 0; green, 0; blue, 0 }  ][line width=0.75]    (10.93,-3.29) .. controls (6.95,-1.4) and (3.31,-0.3) .. (0,0) .. controls (3.31,0.3) and (6.95,1.4) .. (10.93,3.29)   ;
\draw    (406.67,126) -- (447.6,167.57) ;
\draw [shift={(449,169)}, rotate = 225.45] [color={rgb, 255:red, 0; green, 0; blue, 0 }  ][line width=0.75]    (10.93,-3.29) .. controls (6.95,-1.4) and (3.31,-0.3) .. (0,0) .. controls (3.31,0.3) and (6.95,1.4) .. (10.93,3.29)   ;
\draw    (272,245.33) -- (414.33,245.66) ;
\draw [shift={(416.33,245.67)}, rotate = 180.13] [color={rgb, 255:red, 0; green, 0; blue, 0 }  ][line width=0.75]    (10.93,-3.29) .. controls (6.95,-1.4) and (3.31,-0.3) .. (0,0) .. controls (3.31,0.3) and (6.95,1.4) .. (10.93,3.29)   ;

\draw (116.67,119) node [anchor=north west][inner sep=0.75pt]   [align=left] {$\displaystyle C_{0}$};
\draw (25.33,240.33) node [anchor=north west][inner sep=0.75pt]  [font=\footnotesize] [align=left] {$\displaystyle T$};
\draw (55.33,150.67) node [anchor=north west][inner sep=0.75pt]  [font=\footnotesize] [align=left] {$\displaystyle X$};
\draw (186,176.67) node [anchor=north west][inner sep=0.75pt]  [font=\footnotesize] [align=left] {$\displaystyle R_{1}$};
\draw (210.67,192) node [anchor=north west][inner sep=0.75pt]  [font=\footnotesize] [align=left] {$\displaystyle R_{2}$};
\draw (400,18.33) node [anchor=north west][inner sep=0.75pt]   [align=left] {$\displaystyle C'_{0}$};
\draw (555.33,149) node [anchor=north west][inner sep=0.75pt]   [align=left] {$\displaystyle \overline{C}_{0}$};

\end{tikzpicture}

    \caption{The curve $C_0$ along with the contraction $\widebar{C_0}$.}
    \label{fig:2}
\end{figure}

\appendix
\section{Invariance of Residue (by Adrian Neff and Jonathan Wise)}\label{inv}

In order to show that Definition \ref{residue} is invariant of the choice of local coordinate $x$, we will need to split into the cases of positive and zero characteristic. Fix a local artinian ring $A$ and let $A(\!(x)\!)$ be the ring of formal Laurent series over $A$. Give $A(\!(x)\!)$ the topology whose open neighborhoods of zero have a base consisting of $A$-submodules of the form $x^iA[\![x]\!]$ for $i \in \Z$. We begin with some generalities on continuous automorphisms of $A(\!(x)\!)$.

\subsection{Automorphisms of Laurent Series}

We begin by classifying continuous automorphisms of $A(\!(x)\!)$. 

Let $\varphi : A(\!(x)\!) \to A(\!(x)\!)$ be a continuous automorphism. Since $x \in A(\!(x)\!)^*$, we must have $\phi(x) \in A(\!(x)\!)^*$. Thus, we must classify the group of units $A(\!(x)\!)^*$. We introduce the following subset of $A(\!(x)\!)$:
\[
A(\!(x)\!)^{nil} = \left\{\sum a_ix^i \, | \, a_0 \in A^*, a_n \text{ is nilpotent for } n < 0\right\}.
\]

\begin{lemma}\label{units}
    Let $A$ be a local artinian ring. Then $A(\!(x)\!)^*$ is in bijection with $\Z \times A(\!(x)\!)^{nil}$.
\end{lemma}

\begin{proof}
    Indeed, since $\Spec(A)$ is connected, a unit in $A(\!(x)\!)$ is of the form $\sum a_ix^i$ with $a_N \in A^*$ for some $N$, and $a_n$ nilpotent for $n < N$, which can be written as $x^N(\sum a_ix^{i-N})$ with $\sum a_ix^{i-N} \in A(\!(x)\!)^{nil}$ (see, e.g., Lemme 0.7 of \cite{JacobienneLocal}). We then define the bijection by 
    \[
    x^N\left(\sum a_ix^{i-N}\right) \mapsto \left(N,\sum a_ix^{i-N}\right).
    \]
\end{proof}

Following the notation of \cite{LaurentHoms}, for a unit $u \in A(\!(x)\!)^*$, denote $\nu(u) = N$, where this is the associated $N$ from Lemma \ref{units}. Explicitly, if 
\[
u = \sum a_ix^i,
\]
then $\nu(u)$ is the $N$ such that $a_N \in A^*$ and $a_n$ is nilpotent for $n < N$. It is shown in \cite{LaurentHoms} (Theorem 4.7 with $n = m = 1$) that every continuous endomorphism $\varphi : A(\!(x)\!) \to A(\!(x)\!)$ satisfies $\nu(\varphi(x)\!) > 0$. We now classify which of these are automorphisms.

\begin{proposition}\label{autos}
    A continuous endomorphism $\varphi : A(\!(x)\!) \to A(\!(x)\!)$ is an automorphism if and only if $\nu(\varphi(x)\!) = 1$.
\end{proposition}

\begin{proof}
    This follows immediately from Theorem 6.8 (with $n = 1$) of \cite{LaurentHoms}. However, we provide a proof for this case here.

    As explained above, $\varphi$ must satisfy $\nu(\varphi(x)\!) > 0$. Let us examine the image of $\varphi(x)$ under the map $A(\!(x)\!) \to k(\!(x)\!)$ induced by $A \to A/\mm \simeq k$. Denote this image by $\overline{\varphi(x)}$. Since $\nu(\varphi(x)\!) > 0$, we see that $\overline{\varphi(x)} = \overline{a_1}x + \overline{a_2}x^2 + \cdots \in xk[\![x]\!]$. On the other hand, if $\overline{\varphi(x)} \in x^2k[\![x]\!]$, i.e., if $\overline{a_1} = 0$, then this will not be an invertible map, as it will not be surjective ($x$ will not be in the image). Thus, we see that $\varphi$ is an automorphism if and only if $\nu(\varphi(x)\!) = 1$.
\end{proof}

\subsection{Characteristic Zero Residue}

Assume that $k$ has characteristic zero. 

\begin{lemma}\label{quotient}
    We have an isomorphism
    \[
    A(\!(x)\!)\,dx/dA(\!(x)\!) \simeq A \, x^{-1}\,dx = A \, d\log x.
    \]
\end{lemma}

\begin{proof}
    To see this, we define a surjective map $A(\!(x)\!)\,dx \to A \, x^{-1}\,dx$ by $\sum a_ix^i\,dx \mapsto a_{-1}x^{-1}\,dx$. We claim that the kernel is exactly $dA(\!(x)\!)$. Indeed, if we have an element $\sum a_ix^i\,dx \in A(\!(x)\!)\,dx$ with $a_{-1} = 0$, then this element is the image under $d$ of $\sum a_ix^{i+1}/(i+1)\,dx$.
\end{proof}

\begin{proposition}\label{char-zero}
    The isomorphism $A(\!(x)\!)\,dx/dA(\!(x)\!) \simeq A \, d\log x$ is invariant under a continuous automorphism of $A(\!(x)\!)$. 
\end{proposition}

\begin{proof}
    From Proposition \ref{autos}, any continuous automorphism must send $x$ to $xu$, where $u \in A(\!(x)\!)^{nil}$. We must show that $d\log(xu) \equiv d\log(x) \pmod{dA(\!(x)\!)}$. Since $d\log(xu) = d\log(x) + d\log(u)$, this amounts to showing that $d\log(u) \in dA(\!(x)\!)$. 

    First, up to scaling by an element of $A^*$, we can write $u = 1 + v$, where $v = \sum a_ix^i$ with $a_0 = 0$ and $a_n$ nilpotent for $n < 0$. Now let $F(x) = \sum_{i > 0} (-1)^{i+1}x^i/i$ be the power series of $\log(1+x)$. We claim that $F(v)$ is a well-defined formal Laurent series. To see this, note that we can write $v = v^{nil} + v'$, where $v^{nil} = \sum_{i < 0} a_ix^i$ (with each $a_i$ nilpotent) and $v' = \sum_{i > 0} a_ix^i$. Since $v^{nil}$ is nilpotent, say with $(v^{nil})^N = 0$, we see that for any $i > 0$, $v^i$ is a sum of $(v^{nil})^j(v')^k$ with $j+k = i$ and $i \geq k > i-N$. Thus for any $n$, there will only be finitely many $i$ such that $v^i$ contributes a nonzero coefficient of $x^n$. 

    Next, $d(F(x)\!) = \sum_{i > 0} (-1)^ix^i \, dx= (1+x)^{-1} \, dx$, and since $F(v)$ is defined, we have
    \[
    d(F(v)\!) = dF(v) \cdot dv = (1+v)^{-1}dv = d(1+v)/(1+v) = du/u = d\log(u).
    \]
    Therefore, $d\log(u)$ is the image of $F(v)$ under $d$, so $d\log(u) \in dA(\!(x)\!)$, as desired.
\end{proof}

\subsection{Positive Characteristic Residue}

Now assume $k$ has characteristic $p \neq 0$. 

One might be tempted to replicate the proof from characteristic zero, but an issue quickly arises: Lemma \ref{quotient} fails to be true when $k$ has positive characteristic. This is due to the fact that in characteristic $p$, $d(x^{p^n}) = p^n x^{p^n-1} \, dx = 0$ for all $n$. Thus, the quotient $A(\!(x)\!) \, dx/dA(\!(x)\!)$ will be too large, as it will contain terms of the form $x^p$, $x^{p^2}$, etc. 

However, we can still recover the idea from characteristic zero as follows. The goal is to quotient $A(\!(x)\!) \, dx$ and have a free rank $1$ $A$-module. Attempting to quotient by just $dA(\!(x)\!)$ fails, as just discussed, so we will try to quotient by something larger. Specifically, we saw that the terms of the form $x^{p^n}$ are what cause the issue with the quotient, so we will try to construct a submodule that contains these terms, leaving us with just $d\log x$ after quotienting.

For each $n$, we define a derivation
\begin{equation*}
p^{-n} d : A(\!(x^{p^n})\!) \to A(\!(x)\!) \: dx
\end{equation*}
by the following formula:
\begin{equation*}
p^{-n} d(x^{p^n}) = x^{p^n - 1} dx = x^{p^n} d \log x .
\end{equation*}
This extends uniquely to a continuous derivation. Equivalently, $p^{-n} d$ is the composition of the derivation
\begin{equation*}
d : A(\!(x^{p^n})\!) \to A(\!(x^{p^n})\!) d(x^{p^n})
\end{equation*}
that sends $x^{p^n}$ to $d (x^{p^n}) = x^{p^n} d \log(x^{p^n})$ with the homomorphism
\begin{equation*}
	A(\!(x^{p^n})\!) d \log (x^{p^n}) \to A(\!(x^{p^n})\!) d \log x \subset A(\!(x)\!) d \log x
\end{equation*}
that sends the basis element $d \log x^{p^n}$ to $d \log x$ (or, equivalently, sends $d(x^{p^n})$ to $x^{p^n - 1} dx$).

Our goal will now be to show that $\sum_n p^{-n}dA(\!(x^{p^n})\!)$ is the submodule that we should quotient by. Recall that a ring $A$ is \textit{$p$-torsion-free} if there is no element $a \in A$ such that $pa = 0$ (e.g, $\Z_p$).

\begin{lemma}\label{lem:2}
    Let $A$ be a $p$-torsion-free local ring with residue field of characteristic $p$.  Let $r \geq 0$ be an integer.  Let $u = \sum a_n x^n \in A(\!(x)\!)$ be a Laurent series with the property that $\ord_p(a_n) + \ord_p(n) \geq r$ for all $n$.  Then $u^p = \sum b_n x^n$ has the property that $\ord_p(b_n) + \ord_p(n) \geq r + 1$ for all $n$. 
\end{lemma}

\begin{proof}
	Suppose  that $p^k \mid n$ and $p^{k+1} \nmid n$. Then $b_n$ is a sum of terms
	\begin{equation} \label{eqn:1}
		a_{i_1} \cdots a_{i_p} x^{i_1 + \cdots + i_p}
	\end{equation}
	where $i_1 + \cdots + i_p = n$.  Since $i_1 + \cdots + i_p$ is not a multiple of $p^{k+1}$, at least one of the $i_j$ must not be a multiple of $p^{k+1}$.  Therefore $\ord_p(a_{i_j}) \geq r - k$ so $\ord_p(a_{i_1} \cdots a_{i_p}) + \ord_p(n) \geq r - k + k = r$.  But if $i_1, \ldots, i_p$ are not all the same then the term~\eqref{eqn:1} appears a multiple of $p$ times in $u^p$, so~\eqref{eqn:1} and all of its permutations contribute a term $c x^n$ to $u^p$ with $\ord_p(c) \geq r + 1$.

	On the other hand, if $i_1 = \cdots = i_p = i$ then $n = ip$.  Since $p^{k+1}$ does not divide $n$, this means that $p^k$ does not divide $i$ (and $p^{k-1}$ does).  The contribution of~\eqref{eqn:1} to $u^p$ is $a_i^p x^n$ in this case.  But $\ord_p(a_i) + \ord_p(i) \geq r$, so we certainly have $\ord_p(a_i^p) + \ord_p(ip) \geq r + 1$.
\end{proof}

\begin{lemma}\label{lem:3}
    Let $A$ be a $p$-torsion-free local ring. The $A$-submodule
\begin{equation*}
	\sum p^{-n} d A(\!(x^{p^n})\!) \subset A(\!(x)\!) d \log(x)
\end{equation*}
is invariant under continuous automorphisms of $A(\!(x)\!)$.
\end{lemma} 

\begin{proof}
	From Proposition \ref{autos}, a continuous automorphism must send $x$ to $xu$, where $u \in A(\!(x)\!)^{nil}$.  We therefore have to show that $p^{-m} d(x^{p^m} u^{p^m})$ lies in $\sum p^{-n} d A(\!(x^{p^n})\!)$ for all $m$.  We write $xu = \sum a_i x^i$ and compute $(xu)^{p^m}$ formally.  We write $(xu)^{p^m} = \sum b_i x^i$ so that
	\begin{equation*}
		p^{-m} d(x^{p^m} u^{p^m}) = p^{-m} d \sum b_i x^i .
	\end{equation*}
	In order for this sum to lie in $\sum p^{-n} d A(\!(x^{p^n})\!)$, we have to verify that $\ord_p(b_i) + \ord_p(i) \geq m$ for all $i$. This follows immediately from an inductive application of Lemma~\ref{lem:2}.
\end{proof}

It follows from Lemma~\ref{lem:3} that the quotient
\begin{equation*}
	A(\!(x)\!) \, dx \Big/ \sum p^{-n} d A(\!(x^{p^n})\!)
\end{equation*}
is an invariant of the ring $A(\!(x)\!)$ and does not depend on the choice of topological generator $x$.

\begin{lemma}
    The quotient
\begin{equation*}
	A(\!(x)\!) \: dx \Big/ \bigl( d A(\!(x)\!) + p^{-1} d A(\!(x^p)\!) +  p^{-2} d A(\!(x^{p^2})\!) + \cdots \bigr)
\end{equation*}
is a free $A$-module of rank~$1$ with basis $d \log x$.
\end{lemma}

\begin{proof}
$A(\!(x)\!)$ has a topological basis consisting of all $x^m d \log x$.  If $m \neq 0$, we can write $m = p^n q$ with $q$ relatively prime to $p$.  Then 
	\begin{equation*}
		x^m d \log x = x^{p^n q} \frac{dx}{x} = x^{p^n q - 1} dx = p^{-n} d (q^{-1} (x^{p^n})^q) .
	\end{equation*}
	Therefore $x^m d \log x$ is in the image of $p^{-n} d$, so the quotient is spanned by $d \log x$.

On the other hand, it is clear that $A d \log x$ does not meet the image of $p^{-n} d$ for any $n$.
\end{proof}

\begin{lemma}\label{lem:1}
Let $u \in A(\!(x)\!)^{nil}$.  Then $d \log u$ lies in
\begin{equation*}
	d A(\!(x)\!) + p^{-1} d A(\!(x^p)\!) + p^{-2} d A(\!(x^{p^2})\!) + \cdots \subset A(\!(x)\!) dx .
\end{equation*}
\end{lemma}
\begin{proof}
	After multiplication by a unit, we can write $u = 1 + v + w$, where $v \in x^{-1} A[x^{-1}]$ has nilpotent coefficients and $w \in x A[\![x]\!]$.  We write $v = \sum_{i=1}^N a_i x^{-i}$ and $w = \sum b_i x^i$.

	We argue that there is a power series $g \in 1 + x A[\![x]\!]$ such that $gu = 1 + x^{-1} f$ where $f \in A[x^{-1}]$ has nilpotent coefficients.  Writing 
	\begin{equation*}
		g = 1 + c_1 x + c_2 x^2 + \cdots
	\end{equation*}
	for indeterminates $c_1, c_2, \ldots$, we obtain the following equations determined by considering the coefficients of $x, x^2, x^3, \ldots$ in $gu$:
	\begin{equation*}
		\begin{split}
			c_1 + a_1 c_2 + a_2 c_3 + \cdots + a_N c_{N+1} & = -b_1 \\
			b_1 c_1 + c_2 + a_1 c_3 + \cdots + a_N c_{N+2} & = -b_2 \\
			b_2 c_1 + b_1 c_2 + c_3 + a_1 c_4 + \cdots + a_N c_{N+3} & = -b_3 \\
			\vdots \hskip2cm 
		\end{split}
	\end{equation*}
	Here are the same equations in matrix form:
	\begin{equation*}
		\begin{pmatrix}
			1 & a_1 & a_2 & a_3 & \cdots & a_N & 0 & 0 & \cdots\\
			b_1 & 1 & a_1 & a_2 & \cdots & a_{N-1} & a_N & 0 &  \cdots \\
			b_2 & b_1 & 1 & a_1 & \cdots & a_{N-2} & a_{N-1} & a_N & \cdots \\
			b_3 & b_2 & b_1 & 1 & \cdots & a_{N-3} & a_{N-2} & a_{N-1} & \cdots \\
			\vdots & \vdots & \vdots & \vdots & \ddots & \vdots & \vdots & \vdots & \ddots 
		\end{pmatrix}
		\begin{pmatrix}
			c_1 \\ c_2 \\ c_3 \\ c_4 \\ \vdots
		\end{pmatrix}
		= - \begin{pmatrix}
			b_1 \\ b_2 \\ b_3 \\ b_4 \\ \vdots
		\end{pmatrix}
	\end{equation*}
	Applying Gaussian elimination to the coefficient matrix converts it into an upper triangular matrix.  Since all of the $a_i$ are nilpotent, Gaussian elimination only adds nilpotent elements to the diagonal entries, and therefore produces an upper triangular matrix whose diagonal entries are units and whose above-diagonal entries are nilpotent.

	Let $J$ be the ideal of $A$ generated by $a_1, \ldots, a_N$.  Since $J$ is generated by finitely many nilpotent elements, it is nilpotent as an ideal.  After the first stage of Gaussian elimination, our matrix has the following form:
	\begin{equation*}
		\begin{pmatrix}
			1 & J & J & \cdots & J & 0 & 0 & 0 & \cdots \\
			0 & 1 & J & \cdots & J & J & 0 & 0 & \cdots \\
			0 & 0 & 1 & \cdots & J & J & J & 0 & \cdots \\
			\vdots & \vdots & \vdots & \ddots & \vdots & \vdots & \vdots & \vdots & \ddots
		\end{pmatrix}
	\end{equation*}
	After eliminating the entries just above the diagonal, it will have this form:
	\begin{equation*}
		\begin{pmatrix}
			1 & 0 & J & \cdots & J & J^2 & 0 & 0 & \cdots \\
			0 & 1 & 0 & \cdots & J & J & J^2 & 0 & \cdots \\
			0 & 0 & 1 & \cdots & J & J & J & J^2 & \cdots \\
			\vdots & \vdots & \vdots & \ddots & \vdots & \vdots & \vdots & \vdots & \ddots

		\end{pmatrix}
	\end{equation*}
	Each round of elimination will replace a diagonal with entries in $J^m$ with a higher diagonal having entries in $J^{m+1}$.  Since $J$ is nilpotent, this process will eventually terminate and reduce the matrix to the identity.

	Now since $g$ is a unit of $A(\!(x)\!)$, we can now write $u$ as $(1 + v)(1 + w)$ where $v \in x^{-1} A[x^{-1}]$ and has nilpotent coefficients and $w \in x A[\![x]\!]$.  Since $d\log(u) = d\log(1+v) + d\log(1+w)$, we can therefore treat the two cases $u = 1 + v$ and $u = 1 + w$ separately.

	We consider the case where $u \in 1 + x A[\![x]\!]$ first.  Replacing $A$ by $\mathbf Z_{(p)}[[b_1, b_2, \ldots]]$, it is sufficient to assume that $A$ is $p$-torsion-free.
%
%
	Since $d \log u = du / u = \sum c_n x^n d \log(x)$ has a power series expansion with coefficients in $A$, it is possible to write a power series $\log u$ with coefficients in $A[p^{-1}]$,
	\begin{equation*}
		\log u = \sum \frac{c_n}{n} x^n ,
	\end{equation*}
	where all $c_n$ are in $A$.  We group the terms by the power of $p$ appearing in $n$, so that
	\begin{equation*}
		\log u = \sum_{\ord_p n = 0} \frac{c_n}{n} x^n + p^{-1} \sum_{\ord_p n = 1} \frac{c_n}{n/p} (x^p)^{n/p} + p^{-2} \sum_{\ord_p n = 2} \frac{c_n}{n/p^2} (x^{p^2})^{n/p^2} + \cdots .
	\end{equation*}
	The first sum lies in $A(\!(x)\!)$, the second in $A(\!(x^p)\!)$, the third in $A(\!(x^{p^2})\!)$, etc.  Applying $d$, therefore,
	\begin{multline*}
		d \log u = d \sum_{\ord_p n = 0} \frac{c_n}{n} x^n + p^{-1} d \sum_{\ord_p n = 1} \frac{c_n}{n/p} (x^p)^{n/p} + p^{-2} d \sum_{\ord_p n = 2} \frac{c_n}{n/p^2} (x^{p^2})^{n/p^2} + \cdots  \\
		\in d A(\!(x)\!) + p^{-1} d A(\!(x^p)\!) + p^{-2} d A(\!(x^{p^2})\!) + \cdots
	\end{multline*}

	Next we consider $u = 1 + v$ where $v \in x^{-1} A[x^{-1}]$ has nilpotent coefficients.  But $u(x^{-1}) \in 1 + x A[x]$, so by the above, 
	\begin{equation*}
		d\log u(x^{-1}) \in d A[x] + p^{-1} d A[x^p] + p^{-2} d A[x^{p^2}] + \cdots
	\end{equation*}
	Applying the homomorphism $x \mapsto x^{-1}$ from $A[x]$ to $A(\!(x)\!)$, we conclude that
	\begin{equation*}
		d \log u \in d A(\!(x)\!) + p^{-1} d A(\!(x^{p})\!) + p^{-2} d A(\!(x^{p^2})\!) + \cdots
	\end{equation*}
	as required.
\end{proof}

\begin{proposition}
    The isomorphism
    \[
    A(\!(x)\!) \: dx \Big/ \bigl( d A(\!(x)\!) + p^{-1} d A(\!(x^p)\!) +  p^{-2} d A(\!(x^{p^2})\!) + \cdots \bigr) \simeq A d\log x
    \]
    is invariant under any continuous automorphism of $A(\!(x)\!)$.
\end{proposition}

\begin{proof}
	As in the proof of Proposition \ref{char-zero}, assume that the automorphism sends $x$ to $xu$ with $u \in A(\!(x)\!)^{nil}$. Then we must show that  $d \log(xu) \equiv d\log x$ modulo $ d A(\!(x)\!) + p^{-1} d A(\!(x^p)\!) +  p^{-2} d A(\!(x^{p^2})\!) + \cdots$, and since $d\log(xu) = d\log x + d\log u$, we must show that $d\log u \in d A(\!(x)\!) + p^{-1} d A(\!(x^p)\!) +  p^{-2} d A(\!(x^{p^2})\!) + \cdots$. This is exactly the statement of Lemma \ref{lem:1}.
\end{proof}

\bibliographystyle{amsalpha}
\bibliography{references}

\end{document}